\documentclass[a4paper,11pt]{amsart}

\usepackage{a4wide}
\usepackage{amssymb,amsmath,amsfonts} \allowdisplaybreaks
\usepackage{graphicx}
\usepackage{hyperref}

\usepackage{color}
\usepackage{amsthm}
\theoremstyle{plain}
\newtheorem{Theorem}{Theorem}[section]
\newtheorem*{Theorem*}{Theorem}

\newtheorem{Corollary}[Theorem]{Corollary}
\newtheorem{Lemma}[Theorem]{Lemma}
\theoremstyle{definition}
\newtheorem{Definition}[Theorem]{Definition}

\theoremstyle{remark}
\newtheorem{Remark}[Theorem]{Remark}

\renewcommand{\epsilon}{\varepsilon}
\renewcommand{\phi}{\varphi}
\newcommand{\ato}[2]{\genfrac{}{}{0pt}{2}{#1}{#2}}

\DeclareMathOperator{\Tr}{Tr}

\DeclareMathOperator{\diam}{diam}
\DeclareMathOperator{\supp}{supp}
\DeclareMathOperator{\rank}{rank}
\DeclareMathOperator{\id}{id}

\newcommand{\eul}{\mathrm{e}}

\newcommand{\EE}{\mathbb{E}}
\newcommand{\RR}{\mathbb{R}}
\newcommand{\CC}{\mathbb{C}}
\newcommand{\NN}{\mathbb{N}}
\newcommand{\ZZ}{\mathbb{Z}}
\newcommand{\PP}{\mathbb{P}}

\newcommand{\cF}{\mathcal{F}}

\newcommand{\cB}{\mathcal{B}}

\newcommand{\cS}{\mathcal{S}}
\newcommand{\cA}{\mathcal{A}}
\newcommand{\st}{\ \vert \ }

\newcommand{\hm}[1]{\textbf{*}\leavevmode{\marginpar{\tiny%
$\hbox to 0mm{\hspace*{-0.5mm}$\leftarrow$\hss}%
\vcenter{\vrule depth 0.1mm height 0.1mm width \the\marginparwidth}%
\hbox to 0mm{\hss$\rightarrow$\hspace*{-0.5mm}}$\\\relax\raggedright #1}}}
\title[Uniform approximation of the IDS for long-range percolation Hamiltonians]{Uniform approximation of the integrated density of states for long-range percolation Hamiltonians}
 \author[F.~Schwarzenberger]{Fabian Schwarzenberger}
 \address[F.S.]{Fakult\"at f\"ur Mathematik, TU Chemnitz, 09107 Chemnitz, Germany}
 \urladdr{http://www-user.tu-chemnitz.de/$\sim$fabis}

\begin{document}
\begin{abstract}
In this paper we study the spectrum of long-range percolation graphs. The underlying geometry is given in terms of a finitely generated amenable group. We prove that the integrated density of states (IDS) or spectral distribution function can be approximated uniformly in the energy variable. Using this, we are able to characterise the set of discontinuities of the IDS.
\end{abstract}
\maketitle

\section{Introduction}

In this paper we study spectral properties of random graphs given by long-range percolation models. The underlying geometry is induced by a finitely generated amenable group. We measure the distance of the elements in the group in terms of the word metric with respect to some finite and symmetric set of generators. While the vertex set of the graph under consideration consists of all elements of the group and hence is deterministic, the edges are inserted randomly and mutually independently. The probability of the existence of an edge depends on the distance between the incident vertices measured in terms of the word metric. Though our model allows edges of arbitrary length, long edges are very unlikely. 
Notice that for long-range percolation graphs the Laplace operator $\Delta_\omega$ is almost surely unbounded and not of finite range. Here we say that an operator is of finite range if there is constant $R$ such that the matrix elements of the operator equal zero if their distance to the diagonal is larger than $R$.

We are interested in the spectrum of $\Delta_\omega$ or more precisely in properties of the corresponding \emph{integrated density of states} (IDS), also known as the \emph{spectral distribution function}. More precisely we ask whether this function can be approximated via finite volume analogues. 
Let us describe in more detail the problem under consideration. It is well known, that amenability is equivalent to the existence of a F\o lner sequence $(Q_j)$, cf. \cite{Adachi-93}. Restricting the Laplacian $\Delta_\omega$ to an element $Q_j$ gives a finite dimensional matrix, denoted by $\Delta_\omega[Q_j]$. The distribution of the eigenvalues of $\Delta_\omega[Q_j]$ is encoded in the function $n(\Delta_\omega[Q_j]):\RR\to \RR$ which maps each $E\in\RR$ to the number of eigenvalues of $\Delta_\omega[Q_j]$ not exceeding $E$. This is called an \emph{eigenvalue counting function}. Given this construction it is natural to ask wheather (and with respect to which topology) the eigenvalue counting functions converge when $j$ tends to infinity.

Before further elaborating on this question, let us briefly describe the content of the paper. The next section is devoted to give precise definitions of the geometric and probabilistic setting. In fact we present the details of the mentioned long-range percolation model and introduce the class of groups which our theory applies to. An important property that is required is the existence of a F\o lner sequence $(Q_j)$ such that each $Q_j$ is a monotile of the group. 
In Section 3 we prove a result from the theory of large deviations, namely a Bernstein inequality for random variables. This is used to estimate the number of ``long'' edges (i.e. edges of length longer than some constant $R$) which are incident to a certain set of vertices. We are interested in the convergence of functions which describe the spectra of operators restricted to elements of a F\o lner sequence $(Q_j)$. Therefore it is usefull to prove that boundary effects caused by this restriction either vanish for increasing $j$ or appear only with small probability. This is done Section 4 where we use the mentioned estimate on the number of ``long'' edges to verify a weak form of additivity for the eigenvalue counting functions. This will be one of the key tools in the proof of our main result, namely Theorem \ref{theorem_erg}. A special version of this reads as follows
\begin{Theorem*}
  Let $(Q_n)$ be a strictly increasing and tempered F\o lner sequence of monotiles and $F_\omega(Q)=n(\Delta_\omega[Q])$ the eigenvalue counting function. Then there exists a distribution function $N:\RR\to[0,1]$ such that
\[ \left\Vert N - \frac{F_\omega(Q_n)}{\vert Q_n \vert} \right\Vert_\infty \rightarrow 0, \quad n\to \infty \]
for allmost all $\omega \in \Omega$. The function $N$ is called \emph{integrated density of states}.
\end{Theorem*}
Hence we give an answer to the above formulated question concerning the convergence of the eigenvalue counting functions. Notice that as we consider the supremum norm this theorem proves uniform convergence, which goes far beyond the usually shown pointwise convergence. Another important feature of this result is that the limit-function is non-random. This is not surprising once one notes that there is an ergodic theorem in the background.
In last section Theorem \ref{theorem_erg} is applied to investigate the points of discontinuity of the IDS.

Now we compare the content of this paper to results of previous ones. We start with work where properties of long-range percolation graphs have been studied.
In \cite{AizenmanB-87} and \cite{AntunovicV-08a} the authors investigated the size of percolation clusters in the subcritical phase. While the first considered a model on $\ZZ^d$, the latter focused on the more general class of quasi-transitive graphs.
Work which is closely related to ours was done in \cite{AntunovicV-08b}, where the asymptotic behaviour of the IDS was analysed. In fact it was shown that the IDS (corresponding to the graph Laplacian) exhibits exponential behaviour at the bottom of the spectrum.
Another approach to the study of spectral properties was chosen by Ayadi \cite{Ayadi-09}. He investigates the resolvent of operators on long-range percolation graphs via a finite volume analogues and pays special interest to an associated correlation function.

The approximability of the integrated density of states is studied in various instances in the literature. The first seminal results where obtained by Pastur \cite{Pastur-71} and Shubin \cite{Shubin-79} who proved pointwise convergence of the finite volume approximants in the context of ergodic random operators and almost periodic operators defined on the Euclidean space. Based on these results it is nowadays well known that operators defined on $\RR^d$ or $\ZZ^d$ obeying a certain kind of ergodicity give rise to pointwise convergent sequence of approximating functions.
Beside this many similar results have been obtained for more complex geometric settings and operators. See for instance \cite{Sznitman-89,Sznitman-90,AdachiS-93,PeyerimhoffV-02,LenzPV-04}, where periodic Laplace and Schr\"odinger operators on manifolds have been investigated. Related work in the context of periodic graphs has been done in \cite{MathaiY-02,MathaiSY-03,DodziukLMSY-03,Veselic-05b}. 

In the topology of pointwise convergence the existence of the IDS has been mostly obtained for operators of finite range but also for operators which do not obey this property. See for instance \cite{PasturF-92} and references therein, where the authors prove pointwise convergence of the eigenvalue counting functions for certain random, symmetric operators on $\ell^2(\ZZ^d)$ which are not necessarily of finite range.

Note that in all of the above mentioned cases the authors obtained either pointwise convergence on the whole real axis or pointwise convergence continuity points of the IDS. Both types are much weaker than the convergence with respect to the supremum norm.
This is of special interest in the context of quasi-periodic and percolation models where it is known (see \cite{KlassertLS-03} respectively \cite{ChayesCFST-86,Veselic-05b}) that the set of points of discontinuity of the IDS is usually very large and can even be dense in the spectrum.

Recently uniform convergence has been shown for several types of operators and geometries. In \cite{LenzV-09} the authors present a method which applies to a large class of discrete models. Among these are Anderson and quantum percolation models, quasi-crystal Hamiltonians on Delone sets, Harper operators, random hopping models as well as Hamiltonians associated to percolation on tilings. However, the ideas they use are based on the assumption that the underlying operator is of finite range. 
Another method to obtain uniform existence of the IDS has been invented in \cite{LenzS-06} for Delone dynamical systems and the associated
random operators. Here an ergodic theorem for certain Banach space valued functions has been established. This is applicable for eigenvalue counting functions of finite range operators and leads to their convergence with respect to the supremum norm. 
These ideas have been adapted to the case where the underlying space equals $\ZZ^d$ in \cite{LenzMV-08} and later on for Cayley graphs given through amenable groups \cite{LenzSV-10}. The considered operators therein fulfil certain ergodicity properties and are assumed to be of finite range as well.

The last mentioned papers are closely related to the present one as we make use of an ergodic theorem as well and study the same underlying geometry as in \cite{LenzSV-10}. However we go beyond these results in several ways. The most important difference is that we are able to treat operators which are not of finite range. This property has been used in various instances in the above mentioned papers \cite{LenzS-06,LenzMV-08,LenzV-09,LenzSV-10}, as many estimates therein are based on rank estimates for restrictions of the operator in question. To verify similar results for the present model, it proved to be usefull to apply ideas from the theory of large deviations, namely a Bernstein inequality. Roughly spoken this Bernstein inequality makes it possible to show that appropriate rank estimates hold with high probability. 
Another advantage is that we are able to give characterisation the set of points of discontinuity. In fact we prove that this set consists of all eigenvalues of all finite graphs and hence does not contain a transcendental number.

In summary, it can be said that there are several results which prove uniform existence of the IDS for models with finite range operators and there are results where pointwise convergence is shown for operators which are not necessarily of finite range. To the best of our knowledge, this is the first work where uniform convergence of the eigenvalue counting functions is shown for operators which are not of finite range.

\section{The model}
In this paper we consider long-range percolation models on amenable groups. Firstly we describe the group as a metric space and introduce certain definitions. Afterwords we give the details of the random process which generates the graph $\Gamma_\omega$.

 Let $G$ be a finitely generated group, $P$ a finite and symmetric set of generators and $\id$ the unit element in $G$. Every element $g\in G$ can be written as a product $g=p_{1}\cdots p_{L}$, where $p_{i}\in P$, $i=1,\dots,L$. We say that the distance between two distinct elements $g,h\in G$ equals $L$ if and only if $L$ is the smallest number such there are elements $p_1,\dots,p_L\in P$ satisfying $gh^{-1}=p_1\cdots p_L$. This gives a metric which we will denote by $d: G\times G\rightarrow \NN_0$
\[
d(g,h):=\min\{L\in \NN_0\ \vert \ p_0=\id,\ \exists p_1,\dots,p_L\in P \mbox{ such that } p_0 p_1\cdots p_L = gh^{-1}\}.
\] 
For a ball of radius $R$ around an element $x\in G$ we write $B_R(x):=\{g\in G\vert d(g,x)\leq R\}$ and $B_R:=B_R(\id)$ if $x$ equals the unit element $\id$. The set of all finite subsets of $G$ is denoted by $\cF(G)$. Given a set $Q\in \cF(G)$ we define the \textit{diameter} by $\diam (Q):=\max \{d(g,h)\vert g,h\in Q\}$ and use $\vert Q\vert$ for the \textit{cardinality} of $Q$. Furthermore we introduce the following notation related to the boundary of a finite subset $Q\subset G$:
\begin{equation}
\begin{split}
 &\partial_{\mathrm{int}}^R(Q):= \{x\in Q\,\vert\, d(x,G\setminus Q)\leq R\},\quad \partial_{\mathrm{ext}}^R(Q) := \{x\in G\setminus Q\,\vert\, d(x,Q)\leq R\}, \\[1ex]
&\partial^R(Q)      := \partial_{\mathrm{int}}^R(Q)\cup\partial_{\mathrm{ext}}^R(Q) \quad \mbox{and}\quad Q_R:=  Q\setminus \partial^R (Q).
\end{split}
\end{equation}
We use the notations $(Q_j)$ and $(Q_j)_{j\in \NN}$ for a sequence of finite subsets of $G$, where the index $j$ takes values in $\NN$ and for a fixed element $Q_j$ of such a sequence we write $Q_{j,R}$ instead of $(Q_j)_R=Q_j\setminus \partial^R (Q_j)$. It is well known \cite{Adachi-93}, that \textit{amenability} of $G$ is equivalent to the existence of a sequence $(Q_j)_{j\in \mathbb N}$ of finite subsets of $G$ such that
\begin{equation*}
\lim\limits_{j\rightarrow \infty} \frac{\vert SQ_j \setminus Q_j\vert}{\vert Q_j\vert}=0
\end{equation*}
holds. Such a sequence $(Q_j)_{j\in \mathbb N}$ is called \textit{F\o{}lner sequence}. It is easy to show that 
\[
 \lim_{j\rightarrow\infty}\frac{\vert\partial ^R Q_j\vert}{\vert Q_j\vert}
 =
 \lim_{j\rightarrow\infty}\frac{\vert\partial ^R_{\mathrm{int}} Q_j\vert}{\vert Q_j\vert}
 =
 \lim_{j\rightarrow\infty}\frac{\vert\partial ^R_{\mathrm{ext}} Q_j\vert}{\vert Q_j\vert}
 =
 \lim_{j\rightarrow\infty}\frac{1}{\vert Q_j\vert}
 =
 0
\]
holds for each F\o lner sequence $(Q_j)$ and $R>0$.
 Given subsets $Q,T\subset G$ such that $Q$ is finite, we say that $\{Qt\st t\in T\}$ is a \emph{tiling} of the group $G$, if $G$ is the disjoint union of the sets $Qt$, $t\in T$. In this situation $Q$ is called \emph{tile} or more precisely \emph{monotile}. An equivalent formulation is \emph{$Q$ tiles $G$}.
 An assumption on the group $G$ will be the following: there exists a F\o lner sequence $(Q_n)$ such that each element of the sequence is a monotile.

\begin{Remark}
 Let us briefly discuss this assumption. It can be inferred from \cite{OrnsteinW-87} that the class of groups, containing a F\o lner sequence of monotiles, covers all cyclic groups and solvable groups as well as all finite extensions thereof. Thus this condition is fulfilled for all elementary amenable groups.

 Moreover if $G$ contains a sequence of finite index subgroups $(G_n)$ such that one can choose the sequence of associated fundamental domains $(F_n)$ (with respect to $G$) as F\o lner sequence, then obviously $(F_n)$ fulfils the above condition. Weiss proved in \cite{Weiss-01} that such sequences exist in any residually finite, amenable group. Recently Krieger \cite{Krieger-07} weakend up the condition of being residually finite, more precisely he showed that it is enough to assume that there exists a sequence of finite index subgroups $(H_n)$ such that $\bigcap_{n\in \NN} H_n=\{ \id\}$. These considerations show that in particular any group of polynomial growth fits in our framework.
\end{Remark}
A sequence of $(Q_n)$ of finite subsets of $G$ is said to be \emph{tempered} if for some $C>0$ and all $n\in \NN$
\[
\left\vert \bigcup_{k<n}Q_k^{-1}Q_n\right\vert \leq C\vert Q_n\vert
\]
holds. It can be shown, that each F\o{}lner sequence has a tempered subsequence, see e.g. \cite{Lindenstrauss-01}. We call a sequence $(Q_n)$ \emph{strictly increasing} if $|Q_{n+1}|>|Q_n|$ for all $n\in \NN$. Again one can show that each F\o{}lner sequence has a strictly increasing subsequence. As each subsequence of a strictly increasing sequence is strictly increasing as well, this gives that there is a strictly increasing tempered F\o lner sequence in each amenable group.

We continue with describing the randomness. Let the set of vertices $V$ be given by the elements of the group $G$ and let $E$ be the set of edges of the complete undirected graph $\Gamma=\Gamma(V,E)$. Thus an edge $e\in E$ is an unordered pair of vertices $x,y\in V$ which we will denote by $e=[x,y]$. Let an arbitrary element $p=(p(x))_{x\in G}\in \ell^1(G)$ satisfying 
\begin{equation}\label{def_p}
0\leq p(x) \leq 1 \hspace{1cm} \mbox{and}\hspace{1cm} p(x)=p(x^{-1})\hspace{1cm} \mbox{for all } x\in G  
\end{equation}
be given. In order to generate the a random subset $E_\omega\subset E$ by a percolation process we define for each $e\in E$ the probability that the edge $e=[x,y]$ is an element of $E_\omega$ to be equal to $p(xy^{-1})$.

  More precisely we consider the following probability space: the sample space is given by $\Omega=\{0,1\}^E$ the set of all possible configurations. We take $\cA$ to be the $\sigma$-algebra of subsets of $\Omega$ generated by the cylinder sets. Finally we define the product measure $\PP=\prod_{e\in E}\PP_{e}$ where for each $e=[x,y]\in E$ the probability measure $\PP_e$ on $\{0,1\}$ is given by
 \[
  \PP_{e}(\omega(e)=1)=p(xy^{-1})\hspace{1cm}\mbox{and}\hspace{1cm} \PP_{e}(\omega(e)=0)=1-p(xy^{-1}).
 \]
Thus each $\omega\in\Omega$ gives rise to a graph $\Gamma_\omega=(V,E_\omega)$. Now we discuss an alternative definition of the long-range percolation process.
\begin{Remark}
 We introduced the distribution of the probabilities via an arbitrary function $p\in\ell^1(G)$ satisfying (\ref{def_p}). There is an equivalent and in physical communities more common way to do so.
 
 For each pair of vertices $x,y\in G$ let $J_{x,y}$ be a real number such that
\begin{itemize}
 \item $J_{xz,yz}=J_{x,y}$ for all $z\in G$,
 \item $J:=J_x:= \sum_{y\in G} J_{x,y}$ is finite and independent of $x\in G$.
\end{itemize}
We fix $\beta>0$ and declare an edge $[x,y]$ to be open with probability $1-\eul^{-\beta J_{x,y}}$. To see the equivalence to the above definition it suffices to show that    $\sum_{x\in G}p(x)<\infty$ holds if and only if  $\sum_{y\in G} J_{x,y}<\infty$, where $p(xy^{-1})=1-\eul^{-\beta J_{x,y}}$. Using that $1-\eul^{-s}\leq s$ for all $s\in \RR$ one obtains
\[ 
\sum_{x\in G}p(x) 
= \sum_{y\in G}p(xy^{-1}) 
= \sum_{y\in G} 1-\eul^{-\beta J_{x,y}}
\leq \beta \sum_{y\in G} J_{x,y}.\]
To prove the converse direction we apply Taylors formula, which shows that there exists a constant $T>0$ such that
\[
 1-\eul^{-\beta J_{x,y}}
 = \beta J_{x,y} - \sum_{k=2}^\infty \frac{(-\beta J_{x,y})^k}{k!} 
 \geq \frac{1}{2} \beta J_{x,y}
\]
holds for all $x,y\in G$ satisfying $d(x,y)\geq T$. Thus we get
\[
 \sum_{y\in G} J_{x,y}
=\sum_{ \ato{y\in G}{d(x,y)\leq T}} J_{x,y}+  \sum_{\ato{y\in G}{d(x,y)>T} } J_{x,y}
\leq \sum_{\ato{y\in G}{d(x,y)\leq T}} J_{x,y}+\frac{2}{\beta} \sum_{\ato{y\in G}{d(x,y)>T}} \left(1-\eul^{-\beta J_{x,y}}\right)
\leq c  \sum_{x\in G}p(x)
\]
for $c>0$ large enough.
\end{Remark}
Note that by definition $[x,y]=[y,x]$ and $\PP([x,y]\in E_\omega)=\PP([xz,yz]\in E_\omega)=p(xy^{-1})$. Furthermore we get for $x\in G$
\begin{equation}\label{def_epsilon(R)}
\epsilon(R):=\sum_{y\in G\setminus B_R(x)} p(xy^{-1})=\sum_{y\in G\setminus B_R} p(y)
\end{equation}
and $ \lim_{R\rightarrow\infty}\epsilon(R) = 0$ since $p\in \ell^1(G)$.
The next result shows that for almost all realisations the graph $\Gamma_\omega$ is locally finite. To this end we define for given $x\in G$ and $\omega\in \Omega$ the vertex degree of $x$ in $\Gamma_\omega$ by
 \[
  m_\omega(x):=\vert \{y\in G\st [x,y]\in E_\omega\}\vert \in [0,\infty] .
 \]
\begin{Lemma}\label{lemma_locfin}
 There exists a set $\Omega_{\mathrm{lf}} \subset\Omega$ of full measure such that $m_\omega(x)$ is finite for all $x\in G$ and all $\omega\in \Omega_{\mathrm{lf}}$.
\end{Lemma}
\begin{proof}
 Fix an element $x\in G$. For each $y\in G$ we denote by $A_y:=\{ [x,y]\in E_\omega\}$ the event that $x$ and $y$ are adjacent. Since
 \[
  \sum_{y\in G} \PP(A_y) = \sum_{y\in G} p(xy^{-1}) <\infty
 \]
we can apply Borel-Cantelli Lemma which gives that there exists a set $\Omega_x\subset \Omega$ with probability one such that $m_\omega(x)$ is finite for each $\omega\in \Omega_x$. This implies
\begin{eqnarray*}
 \PP(\{\exists x\in G \mbox{ such that } m_\omega(x)=\infty \})&=&\PP\left(\bigcup_{x\in G} \{ m_\omega(x)=\infty \}\right)\\
 &\leq & \sum_{x\in G} \PP\left(\{ m_\omega(x)=\infty \}\right)\\
 &\leq & \sum_{x\in G} \PP\left(\Omega\setminus \Omega_x\right)
\end{eqnarray*}
As $G$ can have only countable many elements the claim follows.
\end{proof}

We go on defining certain random variables. Given an edge $e\in E$ the random variable $X_e(\omega)$ is equal 
to one if $e$ is an element of $E_\omega$ and zero otherwise. 
If an edge is given by a pair of vertices $[x,y]$ it is obvious that 
$X_{[x,y]}=X_{[y,x]}$ and its distribution depends only on the value $xy^{-1}$.

For fixed $R\in \NN$ and a finite subset $Q=\{x_1,\dots ,x_{\vert Q\vert}\}\subset G$ we define random variables $Y_i, i=1,\dots ,\vert Q\vert$ by
\begin{equation}\label{def_Y_i}
 Y_i(\omega)=\sum_{y\in M_i^R } X_{[x_i,y]}(\omega)\hspace{0.5cm}\mbox{where}\hspace{0.5cm} M_i^R:=\{x\in G\st d(x,x_i)>R,\ x\neq x_j\ \forall j<i\}.
\end{equation}
Thus, $Y_i$ is the random variable counting the edges of length larger than $R$, 
being incident to $x_i$ and not counted by any $Y_j, j=1,\dots,i-1$. 
Note that the variables $Y_i$ are independent and Lemma \ref{lemma_locfin} yields $\PP(Y_i=\infty)=0$, $i=1,\dots,|Q|$. Furthermore the distribution functions for these random variables fullfill $ F_{Y_1}(z)\leq F_{Y_i}(z)$ for all $i\in \{1,\dots,|Q|\}$ and all $z\in \RR$.
By equation (\ref{def_epsilon(R)}) the expectation value $\EE(Y_1)$ equals $\epsilon(R)$. 
We denote the centred random variable $Y_i-\EE(Y_i)$ by $\bar Y_i$ for all $i=1,\dots,\vert Q\vert$ 
and set $Y:= Y_1$, $\bar Y:=\bar Y_1$. The aim of Lemma \ref{lemma_prob} is to describe the distribution of the variables $Y_i$.

\begin{Lemma}\label{lemma_prob}
 Let $R\in \NN$, $Q=\{x_1,x_2,\dots,x_{\vert Q\vert}\}\in\cF(G)$ and $Y_i$, $i=1,\dots,|Q|$ be given as above. Then the estimate 
\[
   \PP(Y_i\geq t)\leq c\eul^{-t}
\]
holds for all $t\in \NN$ and all $i=1,\dots,|Q|$, where $c\in \RR$ is given by
\[
 c=\prod_{y\in G} \left(1+ p(y)(\eul-1)\right).
\]
 
\end{Lemma}
\begin{proof}
Let $y\in G$ be arbitrary and set $x:=x_1$ as well as $Y=Y_1$, then 
\[
 \EE(\eul^{X_{[x,y]}}) = p(xy^{-1}) \eul + (1-p(xy^{-1})) \eul^0 = 1+ p(xy^{-1})(\eul-1)
\]
holds. The independence of $X_e$, $e\in E$ implies
\[
 \EE(\eul^Y) = \prod_{y\in G\setminus B_R(x)} \EE (\eul^{X_{[x,y]}}) 
= \prod_{y\in G\setminus B_R(x)} \left(1+ p(xy^{-1})(\eul-1)\right) 
\leq \prod_{y\in G} \left(1+ p(y)(\eul-1)\right)
\]
since $Y=\sum_{y\in G\setminus B_R(x)} X_{[x,y]}$. The product converges to a finite number since 
\[
 \prod_{y\in G} \left(1+ p(y)(\eul -1)\right)
=\exp \left(\sum_{y\in G}\ln(1+p(y)(\eul -1))\right)
\leq \exp \left((\eul -1) \sum_{y\in G} p(y)\right) <\infty
\]
holds by assumption on $p$. Now we use Markov's inequality to obtain for given $i\in\{1,\dots,|Q|\}$
\[
  \PP(Y_i\geq t)
\leq \PP(Y\geq t)
\leq \eul^{-t} \EE(\eul^Y),
\]
which implies the claimed inequality with constant $c$ not depending on $R$.
\end{proof}
Lemma \ref{lemma_prob} implies that for each $k\in\NN$ and $i\in\{1,\dots,\vert Q\vert\}$ the moments $\EE (Y_i^k)$ and $\EE (\bar Y_i^k)$ exist. This is clear from
\[
 |\EE(Y_i^k)|
= \sum_{t=0}^\infty t^k \PP(Y_i = t)
\leq \sum_{t=0}^\infty t^k \PP(Y_i \geq t)
\leq c \sum_{t=0}^\infty t^k \eul^{-t} 
< \infty
\]
and
\[
 |\EE(\bar Y_i^k)|
=\left| \sum_{t=0}^\infty (t-\EE (Y_i))^k \PP(Y_i = t)\right|
\leq \sum_{t=0}^\infty \left|t-\EE (Y_i)\right|^k \PP(Y_i \geq t)
\leq c \sum_{t=0}^\infty |t-\EE (Y_i)|^k \eul^{-t} 
<\infty.
\]

\section{Bernstein inequality}
In this section we verify a Bernstein inequality for independent random variables $\xi_i$. This is a result from the theory of large deviations. It estimates the probability that the sum of the random variables differs too much from its expectation value. The proof follows ideas from \cite{ArakZ-88} where similar estimates are shown.

\begin{Theorem}[Bernstein inequality] \label{theorem_bernstein}
 Let $\xi_1,\dots,\xi_n$ be independent random variables satisfying 
 \begin{equation}\label{bern_cond}
  \EE(\xi_i)=0\quad \mbox{and}\quad |\EE(\xi_i^k)|\leq \frac{1}{2} \tau^{k-2}k!
 \end{equation}
 for all $i=1,\dots,n$, all $k\in\NN\setminus\{1\}$ and some constant $\tau>0$.
 Then
\[
\PP (S \geq \alpha)
\leq 
\left\{\begin{array}{ll} \eul^{-\frac{\alpha^2}{   4n}} &, 0\leq \alpha \leq n/\tau\\ \eul^{-\frac{\alpha}{4\tau}} &, \alpha > n/\tau \end{array}\right. ,
\]
where $S=\sum_{i=1}^n \xi_i$.
\end{Theorem}
\begin{proof}
 For fixed $i\in\{1,\dots,n\}$ and $h\in(0,\frac{1}{2\tau}]$ we have by assumption on $\xi_i$
\[
 \EE (\eul^{h \xi_i})
= \sum_{k=0}^\infty \frac{\EE((h\xi_i)^k)}{k!} 
\leq 1+ h^2 \sum_{k=2}^\infty h^{k-2}\frac{|\EE(\xi_i^k)|}{k!}
\leq 1+ \frac{h^2}{2} \sum_{k=2}^\infty (h\tau)^{k-2}
\leq 1+h^2
\leq \eul^{h^2}.
\]
Furthermore the independency of the random variables implies
\[
 \EE (\eul^{h S} )
= \prod_{i=1}^n \EE (\eul^{h\xi_i} )
\leq \prod_{i=1}^n  \eul^{h^2}
=   \eul^{nh^2}.
\]
Using this and Markov inequality we obtain
\begin{equation}\label{thm_bern1}
 \PP(S\geq \alpha) \leq \eul^{-\alpha h} \EE(\eul^{h S}) \leq \eul^{nh^2-\alpha h} 
\end{equation}
for each $\alpha>0$. In the case $0<\alpha\leq \frac{n}{\tau}$ set $h=\frac{\alpha}{2n}\leq \frac{1}{2 \tau}$. Then \eqref{thm_bern1} can be written as
\[
\PP(S\geq \alpha) \leq \eul^{-\frac{\alpha^2}{4n}}.
\]
If $\alpha \geq \frac{n}{\tau}$ we set $h=\frac{1}{2\tau}$ an conclude
\[
\PP(S\geq \alpha) \leq \eul^{-\frac{\alpha}{4\tau}},
\]
which proves the claimed estimate.
\end{proof}
 The next Lemma shows that the variables $Y_i$, $i=1,\dots,\vert Q\vert$ fulfil the conditions \eqref{bern_cond} with some parameter $\tau>0$, which is independent of $R$ and $Q$. This allows to apply Theorem \ref{theorem_bernstein} in order to prove an adapted inequality in Corollary \ref{corollary_bernstein}.

\begin{Lemma}\label{lemma_berstein_ineq}
 There exists an $R_0\in \NN$ such that for each $R\geq R_0$ the following holds: for any set $Q=\{x_1,\dots,x_{\vert Q\vert}\}\in \cF(G)$ and associated random variables $Y_i, i=1,\dots,\vert Q\vert$ given as in (\ref{def_Y_i}) each $\bar Y_i=Y_i-\EE (Y_i)$ satisfies the conditions \eqref{bern_cond} with $\tau=6 \prod_{y\in G} \left(1+ p(y)(\eul -1)\right) $.
 \end{Lemma}

\begin{proof}
Notice that the existence of the moments $\EE (\bar Y_i^k)$, $k\in\NN$, $i\in\{1,\dots,|Q|\}$ is already clear from Lemma $\ref{lemma_prob}$. However it is not obvious that the conditions \eqref{bern_cond} hold with $\tau$ given as above. Furthermore we see $\tau=6c$, where $c$ is the constant given by Lemma \ref{lemma_prob}. In the special case where the second moment of $\bar Y_i$ equals zero, the conditions \eqref{bern_cond} are clearly fulfilled since then $\EE (\bar Y_i^k) =0$ for all $k\in \NN$, $i\in\{1,\dots,\vert Q\vert\}$.

Let $Q=\{x_1,\dots,x_{\vert Q\vert}\}\in \cF(G)$, $i\in \{1,\dots,\vert Q\vert \}$ and set $x:=x_1$. We firstly choose a certain constant $T\in \NN$ and give a condition for $R_0$ and in order to prove that $\EE (\bar Y_i^2)$ does not exceed one for all $i=1,\dots,n$ and all $R\geq R_0$. Let $T\in \NN$ be such that
\begin{equation}\label{lemma_bernstein_ineq2}
 \sum_{t=T+1}^\infty t^2 \eul^{-t}\leq \frac{1}{3c},
\end{equation}
where $c>0$ is the constant given by Lemma \ref{lemma_prob}. Now choose $R_0\in \NN$ such that
\begin{equation}\label{lemma_bernstein_ineq3}
 \epsilon(R)\leq -\frac{1}{2} \ln \left( 1- \left( 3 \sum_{t=1}^T t^2 \right)^{-1} \right)
\end{equation}
for all $R\geq R_0$. This choice implies
\begin{equation}\label{lemma_bernstein_ineq4}
\EE (Y_i) \leq \EE (Y) = \epsilon (R)\leq \frac{1}{3} \quad \mbox{and}\quad  p(y)\leq \frac{1}{2} \quad \mbox{for all }R\geq R_0, y\in G\setminus B_{R_0}.
\end{equation}
Furthermore we get for $R\geq R_0$
\[
 \PP(Y_i=0)
\geq \PP(Y=0)
 =   \PP\left(\sum_{y\in G\setminus B_R(x)} \!\!\!\! X_{[x,y]} =0\right)
 =   \prod_{y\in G\setminus B_R(x)} \!\!\!\!(1-p(xy^{-1}))
 =   \prod_{y\in G\setminus B_R } (1-p(y)).
\]
Now we use the inequality $1-z\geq \eul^{-2z}$ which holds for all $z\in [0,0.5]$ and obtain
\[
  \prod_{y\in G\setminus B_R } (1-p(y))
= \exp\left(\sum_{y\in G\setminus B_R } \ln (1-p(y))\right)
\geq \exp\left(-2\sum_{y\in G\setminus B_R } p(y)\right)
= \exp \left(-2 \epsilon(R)\right),
\]
which shows using \eqref{lemma_bernstein_ineq3}
\begin{equation}\label{lemma_bernstein_ineq5}
 \PP(Y_i\geq 1)=1-\PP(Y_i=0)\leq 1-\exp(-2 \epsilon(R))\leq \left(3 \sum_{t=1}^T t^2 \right)^{-1}
\end{equation}
As $\EE (\bar Y_i^2)$ can be written as
\[
 \EE (\bar Y_i^2 )
= \left\vert \EE \left((Y_i-\EE (Y_i))^2\right) \right\vert 
= \sum_{t=0}^\infty (t-\EE (Y_i))^2 \PP(Y_i=t)
\]
the estimates in \eqref{lemma_bernstein_ineq2},\eqref{lemma_bernstein_ineq4},\eqref{lemma_bernstein_ineq5} and Lemma \ref{lemma_prob} imply
\begin{align*}
\EE(\bar Y_i^2) &\leq (\EE (Y_i))^2 + \sum_{t=1}^T (t-\EE (Y_i))^2 \PP(Y_i=t) +\sum_{t=T+1}^\infty (t-\EE (Y_i))^2 \PP(Y_i=t)\\
 & \leq (\epsilon(R))^2  + \PP(Y_i\geq 1) \sum_{t=1}^T t^2 + \sum_{t=T+1}^\infty t^2 \PP(Y_i \geq t)\\
 & \leq \frac 1 3 + \frac 1 3 + \frac 1 3 =1.
\end{align*}

Now let $k\geq 3$. The $k$-th moment of $\bar Y_i$ is by definition the $k$-th central moment of $Y_i$ thus we get
\[
\left\vert \EE (\bar Y_i^k) \right\vert
= \left\vert \EE ((Y_i-\EE (Y_i))^k) \right\vert 
= \left\vert \sum_{t=0}^\infty  (t-\EE (Y_i))^k \PP(Y_i=t)\right\vert.
\]
Since $0\leq \EE (Y_i) \leq \frac 1 3$, see \eqref{lemma_bernstein_ineq2} we have that
\begin{eqnarray*}
\left\vert \sum_{t=0}^\infty  (t-\EE (Y_i))^k \PP(Y_i=t)\right\vert 
&\leq&  (\EE (Y_i))^k \PP(Y_i=0) + \sum_{t=1}^\infty  t^k \PP(Y_i=t)\\
&\leq&   (\EE (Y_i))^k   + \sum_{t=1}^\infty  t^k \PP(Y_i\geq t)
\end{eqnarray*}
holds. Using $\PP(Y_i\geq t)\leq \PP(Y\geq t)$ and $\EE(Y_i)\leq \EE(Y)$, this implies
\begin{equation*}
 \left\vert \EE (\bar Y_i^k) \right\vert 
\leq (\EE (Y))^k  + c\sum_{t=1}^\infty  t^k \eul^{-t},
\end{equation*}
where the last inequality holds with constant $c>0$ from the Lemma \ref{lemma_prob}. The function $f:[0,\infty]\rightarrow \RR$, $x\mapsto x^k \eul^{-x}$ takes its maximal value at the argument $x=k$. Therefore we get
\begin{eqnarray*}
\sum_{t=1}^\infty  t^k  \eul^{-t}
&=& \sum_{t=1}^{ k - 1}  t^k  \eul^{-t}   +  k^k \eul^{- k} +\sum_{t= k+1}^\infty  t^k  \eul^{-t}\\
&\leq& \int_0^{ k}x^k \eul^{-x} dx + k^k \eul^{-k} + \int_{ k}^\infty x^k \eul^{-x} dx \\
&=& \int_0^{\infty}x^k \eul^{-x} dx + k^k \eul^{-k} .
\end{eqnarray*}
Partial integration leads to
\[
 \int_0^{\infty}x^k \eul^{-x} dx = \int_0^{\infty} k x^{k-1}  \eul^{-x} dx =\dots= k!\int_0^{\infty} \eul^{-x} dx = k!
\]
Now it is enough to show that
\begin{equation}\label{lemma_berstein_ineq1}
2(\EE (Y))^k+ 2ck!+2c\left(\frac{k}{\eul}\right)^k \leq \tau^{k-2} k!
\end{equation}
holds for $\tau=6c$.
To this end we consider the three summands separately. The first one gives by \eqref{lemma_bernstein_ineq2} and as $\tau >1$
\[
 \frac{2(\EE (Y))^k}{\tau^{k-2} k!}\leq  \frac{1}{3}.
\]
The second summand gives
\[
\frac{2ck!}{\tau^{k-2} k!}=\frac{2c}{(6c)^{k-2} }\leq \frac 1 3
\]
and for the third summand we use Stirling formula $k!\geq k^k \eul^{-k} $ to obtain
\[
\frac{2c k^k}{\eul^k \tau^{k-2} k!}\leq \frac{2c}{{(6c)}^{k-2}}\leq \frac 1 3.
\]
This shows that \eqref{lemma_berstein_ineq1} holds, which finishes the proof.
\end{proof}

Given a finite set $Q=\{x_1,\dots,x_{\vert Q\vert}\}\subset G$ we will use this result to show that the probability that ``to many long edges'' are incident to a vertex in $Q$ is very small. To be precise, let $R\in\NN$ and $\delta >0$ be constants and set $\epsilon=\epsilon(R)=\EE (Y)$ as in (\ref{def_epsilon(R)}). We decompose the probability space $\Omega=\Omega_1(\delta,R,Q)\cup\Omega_2(\delta,R,Q)$ by setting
\begin{equation}\label{def_Omega_1}
\Omega_1(\delta,R,Q)
:=\left\{\omega\in \Omega\ \bigg\vert\ \sum_{i=1}^{\vert Q\vert} Y_i(\omega)\geq \vert Q\vert (\epsilon+\delta)\right\}
\ \mbox{ and }\ 
\Omega_2(\delta,R,Q)
:=\Omega\setminus \Omega_1(\delta,R,Q).
\end{equation}
where $Y_i$, $i=1,\dots,\vert Q\vert$ are given by (\ref{def_Y_i}). Thus the set $\Omega_1(\delta,R,Q)$ consists of all configurations where the number of edges of length longer than $R$ that are incident to at least one vertex in $Q$ is at least $\vert Q\vert (\epsilon(R) +\delta)$.

\begin{Corollary}\label{corollary_bernstein}
Let $R_0$ and $\tau$ be as in Lemma \ref{lemma_berstein_ineq}, let $\delta>0$ and $Q\in\cF(G)$ be given and define $\Omega_1(\delta,R,Q)$ as above. Then the following inequality holds
\begin{equation}\label{corollary_berstein1}
\PP (\Omega_1(\delta,R,Q))\leq \left\{\begin{array}{ll}\exp\left(-\frac{\delta^2 \vert Q\vert}{ 4 }\right) &, 0\leq \delta \leq \frac{1}{\tau}\\ \exp\left(-\frac{\delta \vert Q\vert} {4\tau}\right) &, \delta > \frac 1 \tau \end{array}\right. .
\end{equation}
\end{Corollary}

\begin{proof}
 By definition of $Y_i$, $\bar Y_i$ and $\epsilon=\epsilon(R)$ we have
\[
\PP(\Omega_1(\delta,R,Q))
= \PP\left( \sum_{i=1}^{\vert Q\vert} Y_i \geq \vert Q\vert (\EE (Y)+\delta)  \right)
\leq \PP\left( \sum_{i=1}^{\vert  Q\vert}\bar Y_i\geq \vert Q\vert \delta  \right).
\]
As the variables $\bar Y_i, i=1,\dots,\vert Q\vert$ are independent and fulfil conditions \eqref{bern_cond} this term can be estimated using Theorem \ref{theorem_bernstein}. Setting $\alpha=\delta\vert Q\vert $ we get
\[
 \PP(\Omega_1(\delta,R,Q))\leq \left\{\begin{array}{ll}\exp\left(-\frac{\delta^2\vert Q\vert^2}{4 |Q|}\right) &, 0\leq \delta\vert Q\vert \leq \frac{|Q|}{\tau}\\ \exp\left(-\frac{\delta \vert Q\vert }{4\tau}\right) &, \delta \vert Q\vert > \frac{|Q|}{\tau} \end{array}\right.  ,
\]
which gives the desired estimate.
\end{proof}

\section{Counting eigenvalues}
In this section we consider the Laplace operator $\Delta_\omega$ with respect to the random graph $\Gamma_\omega$ acting a domain $D_\omega\subset \ell^2(G)$. To define this operator in a appropriate sense we restrict ourselves from now on to the set $\Omega_{\mathrm{lf}} \subset\Omega$ with $\PP(\Omega_{\mathrm{lf}})=1$ where $\Gamma_\omega$ is a locally finite graph for all $\omega \in  \Omega_{\mathrm{lf}}$, cf. Lemma \ref{lemma_locfin}.

We denote by $C_c(G)\subset \ell^2(G)$ the dense subset of functions $f: G\to \CC$ with finite support. On this space we define the operator $\tilde\Delta_\omega: C_c(G)\to \ell^2(G)$ by setting
\[
 \tilde\Delta_\omega f (x) := m_\omega(x) f(x)-\sum_{y: [x,y]\in E_\omega} f(y) = \sum_{y:[x,y]\in E_\omega}\left( f(x)-f(y)\right).
\]
It is known that this operator is essentially selfadjoint, cf. \cite{Jorgensen-08,Wojciechowski-09,Weber-10}. Thus there exists a domain $D_\omega$ such that $\Delta_\omega: D_\omega\to \ell^2(G)$ is the unique selfadjoint extension of $\tilde \Delta_\omega$. The operator $\Delta_\omega$ will be called the Laplace operator.
Similarly the Laplacian $\Delta_S:\ell^2(V_S)\to \ell^2(V_S)$ on a finite subgraph $S=(V_S,E_S)$ of the complete graph $\Gamma$ is given by
\[
\Delta_S f(x) = \sum_{y\in V_S :[x,y]\in E_S}\left( f(x)-f(y)\right).
\]
We denote the set of all finite subgraphs of the complete graph $\Gamma$ by $\cS$. The subset of $\cS$ consisting of all subgraphs with vertex set $Q\in \cF(G)$ is called $\cS(Q)$. 
For a subgraph $S=(V_S,E_S)$ of $\Gamma$ and $Q\subset V_S$ the induced subgraph of $S$ on $Q$ is denoted by $S[Q]$, i.e. $S[Q]$ is the graph on vertex set $Q$, where two vertices are adjacent in $S[Q]$ if and only if they are adjacent in $S$. Given a subgraph $S=(V_S,E_S)$ of $\Gamma$ and an element $x\in G$ the \emph{translation of $S$ by $x$} is the graph $Sx$ whose vertex set is $V_{Sx}=V_Sx=\{yx \in G \st y \in V_S\}$ and the edges are $E_{Sx}=\{[y,y']\in E \st [yx^{-1},y'x^{-1}]\in E_S\}$.

In order to define the restriction of the Laplacian on a subset $Q\subset G$, we introduce mappings $p_Q$ and $i_Q$ called projection and inclusion. The support of $u\in \ell^2(G)$ is the set of those $x\in G$, such that $u(x)\neq 0$. We identify $\ell^2(Q)=\{u:Q\rightarrow {\CC} \vert \sum_{x\in Q}\vert u(x)\vert^2 < \infty\}$ with the subspace of $\ell^2(G)$ consisting of all elements supported in $Q$. The map $p_Q:\ell^2(G)\rightarrow \ell^2(Q)$ is given by $u\mapsto p_Q(u)$, where $p_Q(u)(x)=u(x)$ for $x\in Q$. Similarly $i_Q:\ell^2(Q)\rightarrow \ell^2(G)$ is given by 
\[ i_Q(u)(x):=\left\{\begin{array}{ll} u(x) &\mbox{if $x\in Q$}\\ 0& \mbox{else} \end{array}\right. .\]
For given $\omega \in \Omega_{\mathrm{lf}}$ and $S=(V_S,E_S)\in\cS$ we will particularly be interested in restricted operators $p_Q\Delta_\omega i_Q:\ell^2(Q)\rightarrow \ell^2(Q)$ and $p_U i_{V_S} \Delta_S  p_{V_S} i_U :\ell^2(U)\rightarrow \ell^2(U)$, where $Q\subset G$ and $U\subset V_S $ are finite. For this we will use the notation 
\[
\Delta_\omega[Q]:=p_Q\Delta_\omega i_Q \quad \mbox{and}\quad \Delta_S[U]:=p_U i_{V_S} \Delta_S  p_{V_S} i_U .
\]
Note that these operators are symmetric matrices with real entries, hence their eigenvalues are a subset of the real axis. Particularly for given $\omega\in \Omega_{\mathrm{lf}}$, $R\in \NN_0$ and $Q\in \cF(G)$ we will be interested in the difference
\[
 D_\omega^R(Q):=\Delta_{\Gamma_\omega[Q]}[Q_R]-\Delta_\omega[Q_R].
\]

\begin{Definition} 
 Let ${\mathcal B}(\mathbb R)$ be the Banach space of right-continuous, bounded functions $f:\RR\rightarrow \RR$ equipped with supremum norm. For a selfadjoint operator $A$ on a finite dimensional Hilbert space $V$ we define its cumulative eigenvalue counting function $n(A)\in {\mathcal B}(\mathbb R)$ by setting
\[ n(A)(E):= \vert \{ i\in \mathbb N \ \vert\ \lambda_i\leq E \} \vert\]
for all $E \in \mathbb R$, where $\lambda_i, i=1,\dots,\dim V$ are the eigenvalues of $A$, counted according to their multiplicity.
\end{Definition}
The next two lemmata are stated for completeness reason. Their proofs are to be found for example in \cite{LenzS-06,LenzMV-08,LenzSV-10}.
\begin{Lemma}\label{lemma_rank}
 Let $A$ and $C$ be selfadjoint operators in a finite dimensional Hilbert space, then we have
\[\vert n(A)(E)-n(A+C)(E) \vert \leq \rank(C)\]
for all $E\in\mathbb R$.
\end{Lemma}
\begin{Lemma}\label{lemma_4}
 Let $V$ be a finite dimensional Hilbert space and $U$ a subspace of $V$. If $i:U\rightarrow V$ is the inclusion and $p:V\rightarrow U$ the orthogonal projection, we have
\[\vert n(A)(E)-n(pAi)(E)\vert\leq 4 \cdot \rank (1-i\circ p)\]
for all selfadjoint operators $A$ on $V$ and all energies $E\in \mathbb R$.
\end{Lemma}
For given $Q\in \cF(G)$, $R\in\NN_0$, $\omega \in \Omega_{\mathrm{lf}}$ and $S=(V_S,E_S)\in \cS$ we define $F_\omega^R,F_\omega: \cF(G)\to \cB(\RR)$ by 
\begin{equation}\label{def_F}
F_\omega^R(Q):=n(\Delta_\omega[Q_R])\hspace{1cm}\mbox{and}\hspace{1cm} F_\omega(Q):=F_\omega^0 (Q)= n(\Delta_\omega[Q]) .
\end{equation}
as well as $\tilde F^R,\tilde F: \cS\to \cB(\RR)$ by
\begin{equation}\label{def_tilde_F}
 \tilde F^R(S):=n(\Delta_S[ (V_S)_R])\hspace{1cm}\mbox{and}\hspace{1cm} \tilde F(S):=\tilde F^0 (S)= n(\Delta_S) .
\end{equation}

\begin{Lemma}\label{F_bounded}
Let $R\in\NN_0$, $\omega\in \Omega_{\mathrm{lf}}$ and the functions $F^R_\omega:\cF(G)\to\cB(\RR)$ and $\tilde F^R:\cS\to \cB(\RR)$ be given as above. Then the following holds true:
\begin{itemize}
 \item[(i)] the functions $F_\omega^R$ and $\tilde F^R$ are linearly bounded, in fact
\begin{equation*}
 \Vert F_\omega^R(Q) \Vert \leq \vert Q\vert \hspace{1cm}\mbox{and}\hspace{1cm} \Vert \tilde F^R( S)\Vert \leq \vert V_S \vert
\end{equation*}
\item[(ii)] the function $\tilde F^R$ is invariant under translation, i.e. for any $S\in\cS$ and $x\in G$ we have $$\tilde F^R(S)=\tilde F^R(Sx).$$
\end{itemize}
\end{Lemma}
\begin{proof}
 This follows easily from the definition.
\end{proof}

The next results are devoted to prove further properties of these functions for $R\geq R_0$ with $R_0$ from Lemma \ref{lemma_berstein_ineq}. We will not be able to prove these properties for all $\omega\in \Omega_{\mathrm{lf}}$ but only for all $\omega\in \tilde\Omega$ where 
\begin{equation}\label{def_tilde_Omega}
 \tilde \Omega := \tilde \Omega(\delta,R,Q) :=\Omega_2(\delta,R,Q) \cap\Omega_{\mathrm{lf}} \quad\mbox{and}\quad \Omega_2(\delta,R,Q) \mbox{ given as in (\ref{def_Omega_1})}
\end{equation}
By Corollary \ref{corollary_bernstein} we have $\PP(\tilde \Omega)\geq 1-\exp(-\delta^2 |Q|/4)$ for $\delta\leq\tau^{-1}$.
The function $F_\omega^R:\mathcal F(G)\rightarrow \mathcal B(\RR)$, $Q\mapsto F_\omega^R(Q)$ satisfies a weak form of additivity, described in the next 
\begin{Lemma}\label{lemma_add}
 Let $Q\in\mathcal F(G)$, $R\geq R_0$ and $\delta>0$ be given and set $\tilde\Omega=\tilde\Omega(\delta,R,Q)$ as in (\ref{def_tilde_Omega}) and $\epsilon=\epsilon(R)=\sum_{y\in G\setminus B_R} p(y)$ as in (\ref{def_epsilon(R)}). Then for any disjoint sets $Q_i, i=1,\dots k$ with $Q=\bigcup_i Q_i$ the inequality
\[
 \left\Vert F_\omega^R(Q)-\sum_{i=1}^k F_\omega^R(Q_i) \right\Vert 
 \leq 4 \vert Q\vert (\epsilon +\delta) + 4\sum_{i=1}^k \vert \partial^R(Q_i)\vert
\]
holds for all $\omega\in \tilde\Omega$. Here $R_0$ is the constant given in Lemma \ref{lemma_berstein_ineq}.
\end{Lemma}

\begin{proof}
 Let $\omega\in \tilde\Omega$ and disjoint sets $Q_i, i=1,\dots k$ with $Q=\bigcup_i Q_i$ be given. During the proof we will call the edges of length longer than $R$ the \emph{long edges}. For given $U\in \cF(G)$ we define an operator $L_\omega[U]:\ell^2(U)\to\ell^2(U)$ which does only respect the long edges by
\[
(L_\omega[U]f)(x)= -\!\!\! \sum_{\ato{y\in U :[x,y]\in E_\omega}{d(x,y)>R}} \!\!\! f(y)
\]
and use the notation
\[
 \Delta_\omega^L[U]:=\Delta_\omega [U] -L_\omega[U].
\]
As $\omega$ is an element of $\Omega_2(\delta,R,Q)$, the number of long edges in $\Gamma_\omega$ which are incident to a vertex  in $Q$ is less than $\vert Q\vert(\epsilon +\delta)$.
 Hence the matrices $L_\omega[Q]$ and $L_\omega[Q_R]$ contain not more than $2\vert Q\vert(\epsilon +\delta)$ non-zero elements and we get
\[\rank(L_\omega[Q])\leq 2\vert Q\vert(\epsilon +\delta)\hspace{0.5cm}\mbox{and}\hspace{0.5cm}\rank(L_\omega[Q_R])\leq 2\vert Q\vert(\epsilon +\delta).\]
This combined with Lemma \ref{lemma_rank} gives
\begin{equation}\label{lemma_add1}
\Vert n(\Delta_\omega[Q_R])-n(\Delta_\omega^L[Q_R]) \Vert
\leq 
\rank(L_\omega[Q_R]) 
\leq 
2\vert Q\vert(\epsilon +\delta) 
\end{equation}
which immediately implies
\begin{equation*}
 \left\Vert n(\Delta_\omega[Q_R])-\sum_{i=1}^k n(\Delta_\omega[Q_{i,R}]) \right\Vert 
\leq  
2\vert Q\vert(\epsilon +\delta) + \left\Vert n(\Delta_\omega^L[Q_R])-\sum_{i=1}^k n(\Delta_\omega[Q_{i,R}]) \right\Vert .
\end{equation*}
 Here the last term can be estimated by
\begin{eqnarray*}
 \left\Vert n(\Delta_\omega^L[Q_R])-\sum_{i=1}^k n(\Delta_\omega[Q_{i,R}]) \right\Vert 
&\leq&    
\left\Vert n(\Delta_\omega^L[Q_R])-\sum_{i=1}^k n(\Delta_\omega^L[Q_{i,R}]) \right\Vert\\
&&+\left\Vert \sum_{i=1}^k \left( n(\Delta_\omega^L[Q_{i,R}])- n(\Delta_\omega[Q_{i,R}])\right) \right\Vert.
\end{eqnarray*}
We apply again Lemma \ref{lemma_rank} and the fact that $\sum_i \rank(L_\omega[Q_{i,R}])$ is bounded by the number of non-zero elements in $L_\omega[Q]$ as well. This proves the inequality 
\begin{equation}\label{lemma_add2}
 \left\Vert n(\Delta_\omega[Q_R])-\sum_{i=1}^k n(\Delta_\omega[Q_{i,R}]) \right\Vert 
\leq 
4\vert Q \vert(\epsilon +\delta) + 
\left\Vert n(\Delta_\omega^L[Q_R])-\sum_{i=1}^k n(\Delta_\omega^L[Q_{i,R}]) \right\Vert .
\end{equation}
 Now we use a decoupling argument. By definition of $\Delta_\omega^L[\cdot]$ and $L_\omega[\cdot]$ we get
\[
\Delta_\omega^L\left[\bigcup_{i=1}^k Q_{i,R}\right]
=
\bigoplus\limits_{i=1}^k \left(\Delta_\omega^L[Q_{i,R}]\right).
\]
Therefore we can count the eigenvalues of $\Delta_\omega^L[Q_{i,R}]$ for $i=1,\dots, k$ separately
\[
n\left(\Delta_\omega^L\left[\bigcup_{i=1}^k Q_{i,R}\right]\right)
=
\sum\limits_{i=1}^k  n\left( \Delta_\omega^L[Q_{i,R}] \right).
\]
Now we apply Proposition \ref{lemma_4} with $V=\ell^2(Q_R)$ and $U=\ell^2(\bigcup_{i=1}^k Q_{i,R})$. Hence we get
\[
 \left\Vert n(\Delta_\omega^L[Q_R])-\sum_{i=1}^k n(\Delta_\omega^L[Q_{i,R}]) \right\Vert
 =
 \left\Vert n(\Delta_\omega^L[Q_R])-n\left(\Delta_\omega^L\left[\bigcup_{i=1}^k Q_{i,R}\right]\right) \right\Vert
 \leq
 4\sum_{i=1}^k \vert \partial^R Q_{i}\vert
\]
This together with (\ref{lemma_add2}) finishes the proof.
\end{proof}
The next lemma shows that the functions $F_\omega^R$ and $\tilde F^R$ act similarly with high probability.

\begin{Lemma}\label{lemma_F-tildeF}
 Let $Q\in \cF(G) $, $R\geq R_0$ and $\delta>0$ be given and set $\tilde\Omega=\tilde\Omega(\delta,R,Q)$ as in (\ref{def_tilde_Omega}) and $\epsilon=\epsilon(R)=\sum_{y\in G\setminus B_R} p(y)$ as in (\ref{def_epsilon(R)}). Then for any
choice of disjoint sets $Q_i\subset Q$, $i=1,\dots,k$
 \[
  \sum_{i=1}^k \left\Vert F_\omega^R(Q_i)- \tilde F^R (\Gamma_\omega[{Q_i}]) \right\Vert \leq \vert Q\vert (\epsilon +\delta)
 \]
holds for all $\omega\in\tilde\Omega$. Here $R_0$ is the constant given in Lemma \ref{lemma_berstein_ineq}.
\end{Lemma}

\begin{proof}
Let $\omega\in\tilde\Omega$ and disjoint subsets $Q_i$, $i=1,\dots,k$ of $Q$ be given. By definition of $\tilde F^R$, $F_\omega^R$ and $D_\omega^R(\cdot)$ 
\begin{eqnarray*}
 \sum_{i=1}^k \left\Vert F_\omega^R(Q_i)- \tilde F^R (\Gamma_\omega [{Q_i}]) \right\Vert 
 &=& 
 \sum_{i=1}^k \left\Vert n(\Delta_\omega[Q_{i,R}])- n(\Delta_{\Gamma_\omega [{Q_i}]}[ Q_{i,R}])  \right\Vert \\
 &=& 
 \sum_{i=1}^k \left\Vert n(\Delta_\omega[Q_{i,R}])- n(\Delta_\omega[Q_{i,R}]+ D_\omega^R(Q_i))  \right\Vert
\end{eqnarray*}
holds. Lemma \ref{lemma_rank} yields that 
\begin{equation*}
\sum_{i=1}^k \left\Vert n(\Delta_\omega[Q_{i,R}])- n(\Delta_\omega[Q_{i,R}]+ D_\omega^R(Q_i))  \right\Vert
\leq
\sum_{i=1}^k \rank (D_\omega^R(Q_i))
\end{equation*}
Note that $D_\omega^R(Q):\ell^2(Q_{R})\to \ell^2(Q_{R})$ is a diagonal matrix where the entry at $(x,x)$ denotes the number of edges in $\Gamma_\omega$ from $x\in Q_{R}$ to $G\setminus Q$. The sum of these entries is bounded from above by the number of all edges of length longer than $R$ which are incident to some $x\in Q_R$. Therefore we get
 \[
  \sum_{i=1}^k \rank (D_\omega^R(Q_i))
\leq 
\sum_{i=1}^k \Tr (D_\omega^R(Q_i)) 
=
\sum_{i=1}^k \sum_{x\in Q_{i,R}}  D_\omega^R(Q_i)(x,x)
\leq
\sum_{x\in Q_{R}}  D_\omega^R(Q)(x,x)
 \]
As $\omega$ is an element of $\Omega_2$, the right hand side is not larger than $\vert Q\vert (\epsilon+\delta)$.
\end{proof}

\section{Uniform approximation}
At the beginning of this section we introduce some notation concerning frequencies of finite subgraphs in infinite graphs. 
For two graphs $S,S'\in \cS$ the number of occurrences of translations of the graph $S$ in $S'$ is denoted by 
\[
\sharp_S(S'):=\vert \{x\in G \st V_S x\subset V_{S'},\ S'[V_Sx]= Sx \}\vert.
\]
Counting occurrences of graphs along a F\o{}lner sequence $(U_j)_{j\in \mathbb N}$ leads to the definition of frequencies. 
Let $S\in\cS$, $(U_j)_{j\in \mathbb N}$ be a F\o{}lner sequence and let $\Gamma'=(V,E')$ be a subgraph of $\Gamma$ on the full vertex set $V$. If the limit
\[
\nu_S(\Gamma'):=\lim\limits_{j\rightarrow \infty}\frac{\sharp_S(\Gamma'[U_j])}{\vert U_j \vert}
\]
exists we call $\nu_S(\Gamma')$ the \textit{frequency of $S$ in the graph $\Gamma'$ along $(U_j)_{j\in \mathbb N}$}.
Similarly frequencies can be defined for subgraphs which are not (or sparsly) connected to the rest of the graph. Given $R\in\NN$ and a graph $\Gamma'=(V,E')$ on the full vertex set $V$, we say that a graph $S=(V_S,E_S)$ is \emph{$R$-isolated in $\Gamma'$} if $\Gamma[V_S]=S$ and $[g,h]\notin E' $ for all $g\in V_S$, $h\in G\setminus V_S$ satisfying $d(g,h)\geq R$.  Therefore a $1$-isolated graph $S$ has no edge connecting it with the rest of the graph. For a given graph $S=(V_S,E_S)\in\cS$, a set $Q\in\cF(G)$, $R\in \NN$ and $\Gamma'$ as above we write
\[
\sharp_{S,R}(\Gamma',Q):=\left| \{x\in G \st V_S x\subset Q \mbox{ and } Sx \mbox{ is $R$-isolated in } \Gamma'\}\right|
\]
for the number of occurrences of $R$-isolated copies of $S$ in $Q$. The frequency of an $R$-isolated graph $S$ along a F\o lner sequence $(U_j)$ in $\Gamma'$ is defined by
\[
 \nu_{S,R}(\Gamma'):= \lim\limits_{j\rightarrow \infty}\frac{\sharp_{S,R}(\Gamma', U_j)}{\vert U_j \vert},
\]
if the limit exists. 
In the following the graph $\Gamma'$ will always be given by percolation graph $\Gamma_\omega$, $\omega\in\Omega$. However Lemma \ref{lemma_nu} will show that the frequencies $\nu_S(\Gamma_\omega)$ will coincide for almost all $\omega \in \Omega$. The same will hold true for the frequencies $\nu_{S,R}(\Gamma_\omega)$.
 
We define the action $T$ of $G$ on $(\Omega,\cA,\PP)$ by
\begin{equation}\label{def_T}
T: G \times \Omega\rightarrow \Omega,\hspace{2cm} (g,\omega) \mapsto T_g(\omega):=\omega g^{-1} 
\end{equation}
where $\omega g^{-1}\in\Omega$ is given pointwise by
\[
 \omega g^{-1}([x,y])= \omega([xg,yg])\quad \mbox{for all } x,y\in G.
\]
Note that $T$ is an ergodic and measure preserving left-action on $(\Omega,\cA,\PP)$.

\begin{Lemma}\label{lemma_nu}
 Given $R\in \NN$ and a tempered F\o lner sequence $(Q_n)$, there exists a set $\Omega_{\mathrm{fr}} \subset\Omega$ of full measure such that the frequencies $\nu_S(\Gamma_\omega)$ and $\nu_{S,R}(\Gamma_\omega)$ along $(Q_n)$ exist for all $S=(V_S,E_S)\in \cS$ and all $\omega\in \Omega_{\mathrm{fr}}$, in particular
\begin{align*}
 \nu_S &:= \nu_S(\Gamma_\omega) = \prod_{[x,y]\in E_S} p(xy^{-1}) \cdot \prod_{\ato{[x,y]\notin E_S}{x,y\in V_S}} (1-p(xy^{-1}))\\
 \nu_{S,R} &:= \nu_{S,R}(\Gamma_\omega) = \prod_{[x,y]\in E_S} p(xy^{-1}) \cdot \prod_{\ato{[x,y]\notin E_S}{x,y\in V_S}} (1-p(xy^{-1})) \cdot \prod_{\ato{[x,y]\in E, x\in V_S,}{y\notin  V_S, d(x,y)\geq R }} (1-p(xy^{-1}))
\end{align*}
holds.
\end{Lemma}
 To prove this Lemma we cite a special case of a pointwise Ergodic Theorem due to
 Lindenstrauss (Theorem 1.2 in \cite{Lindenstrauss-01})

\begin{Theorem}\label{theorem_linde}
  Let $G$ act from the left on a measure space $(\Omega,\cA,\PP)$ by an ergodic and measure preserving transformation $T$ an let $(Q_n)$ be a tempered F\o{}lner sequence. Then for any $ f\in L^1(\PP)$
  \[
  \lim_{n\rightarrow\infty}\frac{1}{\vert Q_n \vert}\sum_{g\in Q}f(T_g \omega ) =\int f(\omega) d \PP(\omega)
  \]
  holds almost surely.
\end{Theorem}

\begin{proof}[Proof of Lemma \ref{lemma_nu}]
 Let $S=(V_S,E_S)\in \cS$ be a graph such that $\id\in V_S$. We define $A_S=\{\omega \in \Omega\st \Gamma_\omega[V_S]=S\}$ to be the subset of $\Omega$ consisting of all configurations where $\Gamma_\omega$ coincides with $S$ on $V_S$ and we denote the indicator function of $A_S$ by $f_S$. The number of occurrences of $S$ in the graph $\Gamma_\omega[Q_n]$ can be estimated by
 \begin{equation}\label{percineq}
 \sum\limits_{g\in Q_{n,\diam(V_S)}}f_S(\omega g^{-1})\leq\sharp_S(\Gamma_\omega[Q_n])\leq \sum\limits_{g\in Q_{n}}f_S( \omega g^{-1}).
\end{equation}
This proves on the one hand that
\[
 \limsup\limits_{n\rightarrow \infty} \frac{\sharp_S(\Gamma_\omega[Q_n])}{\vert Q_n \vert}
 \leq
 \limsup\limits_{n\rightarrow \infty} \frac{1}{\vert Q_n\vert} \sum\limits_{g\in Q_{n}}f_S(\omega g^{-1})
\]
holds. Using the fact 
\[
 \frac{1}{\vert Q_n\vert} \sum\limits_{g\in \partial_{\mathrm{int}}^R Q_n} f_S(\omega g^{-1})\leq \frac{\vert\partial_{\mathrm{int}}^R Q_n\vert}{\vert Q_n\vert}\rightarrow 0,\ n\rightarrow \infty
\]
for all $R>0$, (\ref{percineq}) also implies
\[
 \liminf\limits_{n\rightarrow \infty}\frac{\sharp_S(\Gamma_\omega[ Q_n])}{\vert Q_n \vert}
 \geq
\liminf\limits_{n\rightarrow \infty} \frac{1}{\vert Q_n\vert} \sum\limits_{ g \in Q_{n,\diam(V_S)}}f_S(\omega g ^{-1})
 =
\liminf\limits_{n\rightarrow \infty} \frac{1}{\vert Q_n\vert} \sum\limits_{ g \in Q_{n}}f_S(\omega g ^{-1}).
\]
Consequently $\nu_S(\Gamma_\omega)=\lim_{n\rightarrow \infty} {\sharp_S(\Gamma_\omega[Q_n])}/{\vert Q_n \vert}$ exists.
 The left action $T$ of $G$ on $(\Omega,\cA,\PP)$ is given by (\ref{def_T}), thus it is measure preserving and ergodic. Therefore Theorem \ref{theorem_linde} gives
 \[
  \nu_S(\Gamma_\omega)
  =\lim_{n\rightarrow\infty}\frac{1}{\vert Q_n \vert}\sum_{g\in Q_n}f_S( \omega g^{-1} ) 
  =\lim_{n\rightarrow\infty}\frac{1}{\vert Q_n \vert}\sum_{g\in Q_n}f_S(T_g \omega ) 
  = \EE (f_S) \hspace{1cm}\mbox{a.s.}
  \]
 The expectation value $\EE (f_S)$ can be computeted via the probabilities given by $p\in \ell^1(G)$ in (\ref{def_p}):
 \[
  \EE (f_S) 
  = \PP(f_S(\omega)=1) 
  = \prod_{[x,y]\in E_S} p(xy^{-1}) \cdot \prod_{\ato{[x,y]\notin E_S}{x,y\in V_S}} (1-p(xy^{-1})).
 \]
 This procedure works in the same way for $\nu_{S,R}(\Gamma_\omega)$. Defining 
\[
A_{S,R}=\{\omega\in\Omega\st \Gamma_\omega[V_S]=S \mbox{ and } S \mbox{ is $R$-isolated in } \Gamma_\omega\}
\]
and $f_{S,R}$ to be its indicator function, we get
\[
  \nu_{S,R}(\Gamma_\omega)
  =\EE (f_{S,R})
  = \prod_{[x,y]\in E_S} p(xy^{-1}) \cdot \prod_{\ato{[x,y]\notin E_S}{x,y\in V_S}} (1-p(xy^{-1})) \cdot \prod_{\ato{[x,y]\in E, x\in V_S,}{ y\notin V_S, d(x,y)\geq R}} (1-p(xy^{-1})).
\]
Here the last product is finite since $p\in \ell^1(G)$.
\end{proof}

 \begin{Theorem}\label{theorem_erg}
 Let $G$ be a finitely generated, amenable group and let $(Q_n)$ be a strictly increasing, tempered F\o lner sequence of monotiles. Let the functions $F_\omega: {\mathcal F}(G)\rightarrow \cB(\RR)$ and $\tilde F: \cS\to\cB(\RR)$ be given as in (\ref{def_F}) and (\ref{def_tilde_F}). Then the following limits
\[N := \lim\limits_{j\rightarrow \infty}\frac{F_\omega(Q_j)}{\vert Q_j \vert}= \lim\limits_{n \rightarrow\infty} \sum_{S\in {\mathcal S}(Q_n)} \nu_S \frac{\tilde F(S)}{\vert Q_n\vert}  \]
exist and are equal almost surely. The function $N$ is called \emph{integrated density of states}.
\end{Theorem}
The proof of this is based on the following Lemma.
\begin{Lemma}\label{lemma_unif}
 Let $G$ be a finitely generated, amenable group and let $(Q_n)$ be a strictly increasing, tempered F\o lner sequence of monotiles. 
 Let $j\in \NN$, $R\geq R_0$ and $0<\delta\leq \tau^{-1}$ be given, where $R_0$ and $\tau$ are constants given by Lemma \ref{lemma_berstein_ineq}. Set $\epsilon=\epsilon(R)=\sum_{y\in G\setminus B_R} p(y)$ as in (\ref{def_epsilon(R)}) and $\Omega_j =  \tilde \Omega(\delta, R, Q_j)\cap \Omega_{\mathrm{fr}}$, where $\tilde\Omega(\delta,R,Q_j)$ is as in (\ref{def_tilde_Omega}) and $\Omega_{\mathrm{fr}}$ as in Lemma \ref{lemma_nu}. The functions $F^R_\omega: {\mathcal F}(G)\rightarrow \cB(\RR)$ and $\tilde F^R: \cS\to\cB(\RR)$ are defined as in (\ref{def_F}) and (\ref{def_tilde_F}). Then the difference
\[
 D_\omega(j,n,R):=\left\Vert \frac{F_\omega^R(Q_j)}{\vert Q_j \vert} -   \sum_{S\in {\mathcal S}(Q_n)} \nu_S \frac{\tilde F^R(S)}{\vert Q_n\vert}\right\Vert.
\]
satisfies the estimate
\begin{equation*}
  D_\omega(j,n,R) 
\leq  
4 \frac{\vert \partial^R Q_n\vert}{\vert Q_n\vert}
+ \left( 4  \frac{\vert\partial^R Q_n\vert}{\vert Q_n\vert}  +1 \right)\frac{\vert \partial^{\diam Q_n} Q_j \vert}{\vert Q_j\vert} 
+ 5  (\epsilon+\delta)
+    \sum_{S\in \cS (Q_n)} \Big \vert \frac{\sharp_S(\Gamma_\omega [Q_j])}{\vert Q_j\vert}-\nu_S \Big \vert 
\end{equation*}
for all $\omega \in \Omega_{j}$ and all $n\in \NN$.
\end{Lemma}
Note that by Corollary \ref{corollary_bernstein} we have $\PP (\Omega_{j})\geq 1-\exp(-\delta^2|Q|/4)$ for all $0<\delta\leq \tau^{-1} $.

\begin{proof}
Let $n\in\NN$ and $\omega\in \Omega_j$ be given. By inserting zeros we estimate the difference $D_\omega(j,n,R)$ in the following way
\begin{eqnarray*}
 D_\omega(j,n,R)&\leq&\left\Vert \frac{F_\omega^R(Q_j)}{\vert Q_j \vert}
- \sum_{\ato{g\in G}{Q_ng\subset Q_j}} \frac{F_\omega^R(Q_ng)}{\vert Q_j\vert \cdot \vert Q_n\vert} \right\Vert\\
&&+   \left\Vert   \sum_{\ato{g\in G}{Q_ng\subset Q_j}} \frac{F_\omega^R(Q_ng)}{\vert Q_j\vert\cdot \vert Q_n\vert}
- \sum\limits_{S\in {\mathcal S}(Q_n)}\frac{\sharp_S( \Gamma_\omega[Q_j])}{\vert Q_j\vert}\frac{\tilde F^R(S)}{\vert Q_n\vert}\right\Vert \\
&& +\left\Vert \sum\limits_{S\in\cS(Q_n)}\frac{\sharp_S(\Gamma_\omega[Q_j])}{\vert Q_j\vert}\frac{\tilde F^R(S)}{\vert Q_n\vert}
-\sum_{S\in\cS(Q_n)} \nu_S \frac{\tilde F^R(S)}{\vert Q_n\vert}\right\Vert 
\end{eqnarray*}
With another application of the triangle inequality this gives
\[D_\omega(j,n,R)\leq D_\omega^{(1)}(j,n,R)+D_\omega^{(2)}(j,n,R)+D_\omega^{(3)}(j,n,R),\]
where
\begin{eqnarray*}
 D_\omega^{(1)}(j,n,R)&:=& \frac{1}{\vert Q_j\vert \cdot \vert Q_n\vert} \Bigg\Vert \sum_{x\in Q_n} F_\omega^R(Q_j)
- \sum_{\ato{g\in G}{Q_ng\subset Q_j}}F_\omega^R(Q_ng) \Bigg\Vert \\
 D_\omega^{(2)}(j,n,R)&:=& \frac{1}{\vert Q_j\vert \cdot \vert Q_n\vert}  \Bigg\Vert   \sum_{\ato{g\in G}{Q_ng\subset Q_j}} F_\omega^R(Q_ng) - \sum\limits_{S\in {\mathcal S}(Q_n)} \sharp_S( \Gamma_\omega[Q_j])  \tilde F^R(S) \Bigg\Vert \\
 D_\omega^{(3)}(j,n,R)&:=&  \sum\limits_{S\in\cS(Q_n)}\left\vert \frac{\sharp_S(\Gamma_\omega[Q_j])}{\vert Q_j\vert}   -\nu_S\right\vert  \frac{\Vert \tilde F^R(S) \Vert}{\vert Q_n\vert} .
\end{eqnarray*}
We use the boundedness of $\tilde F^R(S) $, see Lemma \ref{F_bounded} to obtain
\begin{equation}\label{lemma_unif_bnd_D3}
D_\omega^{(3)}(j,n,R)\leq  \sum_{S\in\cS(Q_n)} \left \vert \frac{\sharp_S(\Gamma_\omega[Q_j])}{\vert Q_j\vert}-\nu_S\right \vert. 
\end{equation}

To estimate the other terms we make use of the tiling property of the set $Q_n$, which gives us that there exists a set $T_n\subset G$ such that $G$ is the disjoint union of the sets $Q_nt$, $t\in T_n$. For fixed $x\in G$ we shift the grid $T_nx=\{tx\st t\in T_n\}$ and get
\[
 G=Gx=\bigcup_{t\in T_n}Q_n t x =\bigcup_{t\in T_n x}Q_n t 
\]
and $Q_nt\cap Q_nt'=\emptyset$ for distinct $t,t'\in T$. This shows that $\{Q_nt\st t\in T_nx\}$ is a tiling of $G$ as well.
Given a set $U\in {\mathcal F}(G)$ and an element $x\in G$, we set
\[W(U,x,n):=\{g \in T_{n}x\st  Q_ng \cap U \neq \emptyset\}\]
and distinguish two types of elements in $W(U,x,n)$
\[I(U,x,n):=\{g \in T_{n}x\st  Q_ng \subset U \} \hspace{0.6cm}\mbox{and}\hspace{0.6cm}\partial (U,x,n):=W(U,x,n) \setminus I(U,x,n).\]
Therefore translations of $Q_n$ by elements of $I(U,x,n)$ are completely contained in $U$ whereas translations of $Q_n$ by elements of $\partial (U,x,n)$ have non-empty intersections with both $U$ and $G\setminus U$. By construction we have the following equality
\begin{equation}\label{lemma_unif1}
 \{g\in G\st Q_n g \subset Q_j \} = \dot\bigcup_{x\in Q_n} I(Q_j,x,n).
\end{equation}
We use the invariance of $\tilde F^R$ under translation, see Lemma \ref{F_bounded} and then (\ref{lemma_unif1}) to obtain
\begin{eqnarray*}
 D_\omega^{(2)}(j,n,R)&=& \frac{1}{\vert Q_j\vert \cdot \vert Q_n\vert}  \Bigg\Vert   \sum_{\ato{g\in G}{Q_ng\subset Q_j}} \left(F_\omega^R(Q_ng) -    \tilde F^R(\Gamma_\omega[Q_n g])\right) \Bigg\Vert \\
 &\leq & \frac{1}{\vert Q_j\vert \cdot \vert Q_n\vert}     \sum_{x\in Q_n} \sum_{g\in I(Q_j,x,n)}\left\Vert F_\omega^R(Q_ng) -    \tilde F^R(\Gamma_\omega[Q_n g]) \right\Vert .
\end{eqnarray*}
As $\omega\in\tilde\Omega(\delta,R,Q_j)$ and as $Q_ng\cap Q_n h=\emptyset$ for distinct $g,h\in I(Q_j,x,n)$, Lemma \ref{lemma_F-tildeF} leads to
\begin{equation}\label{lemma_unif_bnd_D2}
 D_\omega^{(2)}(j,n,R)\leq \frac{1}{\vert Q_j\vert \cdot \vert Q_n\vert}     \sum_{x\in Q_n} \vert Q_j\vert (\epsilon+\delta) = \epsilon + \delta.
\end{equation}

To estimate $D_\omega^{(1)}(j,n,R)$ firstly note that the disjointness of the translates and the fact that $Q_ng\subset \partial^{\diam Q_n}Q_j$ holds for all $g\in \partial(Q_j,x,n)$ imply the following inequalities:
\begin{equation}\label{diamest}
\vert \partial (Q_j,x,n) \vert \cdot  \vert Q_n\vert  \leq \vert \partial^{\diam Q_n} Q_j\vert          \hspace{1cm}\mbox{and}\hspace{1cm}      \vert I (Q_j,x,n) \vert \cdot \vert Q_n\vert \leq \vert Q_j \vert.
\end{equation}
We use again (\ref{lemma_unif1}) to obtain
\begin{equation}\label{lemma_unif2}
 D_\omega^{(1)}(j,n,R) \leq \frac{1}{\vert Q_j\vert \cdot \vert Q_n\vert} \sum_{x\in Q_n} \Bigg\Vert  F_\omega^R(Q_j)
- \sum_{g\in I(Q_j,x,n)}F_\omega^R(Q_ng) \Bigg\Vert 
\end{equation}
and analyse one summand
\begin{align*}
 A_\omega^R(Q_j,x,n) 
 &:=
 \bigg\Vert F_\omega^R(Q_j) - \sum\limits_{g\in I(Q_j,x,n)} F_\omega^R(Q_ng )\bigg\Vert = \bigg\Vert F_\omega^R(Q_j) - \sum\limits_{g\in I(Q_j,x,n)} F_\omega^R((Q_ng)\cap Q_j)\bigg\Vert  \\
&\leq 
\bigg\Vert F_\omega^R(Q_j) -  \sum\limits_{g\in W(Q_j,x,n)} F_\omega^R((Q_n g) \cap Q_j) \bigg\Vert + \bigg\Vert \sum\limits_{g\in \partial(Q_j,x,n)} F_\omega^R((Q_ng) \cap Q_j)   \bigg \Vert
\end{align*}
where the last inequality holds since $W(Q_j,x,n)$ is the disjoint union of $\partial(Q_j,x,n)$ and $I(Q_j,x,n)$. Next we use the weak form of additivity given by Lemma \ref{lemma_add}. This is applicable since $\omega\in\Omega_j\subset \tilde\Omega(\delta,R,Q_j)$ and gives together with the boundedness of $F_\omega^R$ see Lemma \ref{F_bounded} the following
\[
 \begin{split}
A_\omega^R(Q_j,x,n)\leq  4\bigg ( \sum\limits_{g\in I(Q_j,x,n)}\vert \partial^R(Q_n g)\vert +\sum\limits_{g\in \partial(Q_j,x,n)}\vert \partial^R((Q_n g) \cap Q_j)\vert + \vert Q_j\vert(\epsilon+\delta)\bigg ) \\  + \sum\limits_{g\in \partial(Q_j,x,n)} \vert Q_n g  \vert .
 \end{split}
\]
The invariance of $\partial^R(\cdot)$ and $|\cdot|$ under translation and the inequalities (\ref{diamest}) yield
\begin{align*}
A_\omega^R(Q_j,x,n)&\leq 4 \vert \partial^R Q_n \vert \vert I(Q_j,x,n)\vert   +4  \vert \partial^R Q_n \vert \vert \partial(Q_j,x,n)\vert  +  \vert Q_n \vert \vert \partial(Q_j,x,n)\vert+ 4 \vert Q_j\vert (\epsilon+\delta) \\
&\leq  4 \vert \partial^R Q_n \vert\frac{\vert Q_j\vert}{\vert Q_n\vert}+4 \vert \partial^R Q_n \vert \frac{\vert\partial^{\diam Q_n} Q_j\vert}{\vert Q_n\vert}+ \vert \partial^{\diam Q_n} Q_j \vert + 4 \vert Q_j\vert (\epsilon+\delta)   
\end{align*}
which we plug in at (\ref{lemma_unif2}) and obtain
\begin{align}
 D_\omega^{(1)}(j,n,R) &\leq  \frac{1}{\vert Q_j \vert}\left( 4 \vert \partial^R Q_n\vert\frac{\vert Q_j\vert}{\vert Q_n\vert}+\left( 4 \frac{\vert \partial^R Q_n\vert}{\vert Q_n\vert}+ 1\right)\vert \partial^{\diam Q_n} Q_j \vert + 4 \vert Q_j\vert (\epsilon+\delta)   \right) \nonumber \\
 &= 4 \frac{\vert \partial^R Q_n\vert}{\vert Q_n\vert}+ \left( 4  \frac{\vert\partial^R Q_n\vert}{\vert Q_n\vert}  +1 \right)\frac{\vert \partial^{\diam Q_n} Q_j \vert}{\vert Q_j\vert} + 4  (\epsilon+\delta) . \label{lemma_unif_bnd_D1}
\end{align}
The combination of the estimates in (\ref{lemma_unif_bnd_D3}), (\ref{lemma_unif_bnd_D2}) and (\ref{lemma_unif_bnd_D1}) gives
\[
 D_\omega(j,n,R) \leq 4 \frac{\vert \partial^R Q_n\vert}{\vert Q_n\vert}+ \left( 4  \frac{\vert\partial^R Q_n\vert}{\vert Q_n\vert}  +1 \right)\frac{\vert \partial^{\diam Q_n} Q_j \vert}{\vert Q_j\vert} + 5  (\epsilon+\delta)+    \sum_{S\in \cS(Q_n)} \Big \vert \frac{\sharp_S(\Gamma_\omega[Q_j])}{\vert Q_j\vert}-\nu_S\Big \vert
\]
which proves the desired estimate on $D_\omega(j,n,R)$. 
\end{proof}

\begin{proof}[Proof of Theorem \ref{theorem_erg}]
For given $j,n\in \NN$, $R\geq R_0$, $0<\delta\leq \tau^{-1}$ and $\omega \in \Omega_j:= \tilde\Omega(\delta,R,Q_j)\cap\Omega_{\mathrm{fr}}$ we set
\[
 B_\omega(j,n,R,\delta):=4 \frac{\vert \partial^R Q_n\vert}{\vert Q_n\vert}+ \left( 4  \frac{\vert\partial^R Q_n\vert}{\vert Q_n\vert}  +1 \right)\frac{\vert \partial^{\diam Q_n} Q_j \vert}{\vert Q_j\vert} + 5  (\epsilon+\delta)+    \sum_{S\in \cS(Q_n)} \Big \vert \frac{\sharp_S(\Gamma_\omega[Q_j])}{\vert Q_j\vert}-\nu_S\Big \vert
\]
i.e. the upper bound for $D_\omega(j,n,R)$ given in the previous Lemma. In the following we explain how to choose the mutual dependences of the parameters $j,n,R,\delta$ in order to obtain sufficient control on $B_\omega(j,n,R,\delta)$ and $\PP(\Omega_j)$ and be able to conclude the statement of the theorem.

 Since $(Q_n)$ is a F\o lner sequence we have for all $R\in \NN$ that $\lim_{n\rightarrow\infty}\vert Q_n\vert^{-1}\vert \partial^R Q_n \vert=0$. The function $R(n)$ is defined inductively in the following way: for all $k\in\NN$ we choose $n_k$ to be the smallest natural number such that $\vert Q_n\vert^{-1}\vert \partial^k Q_n \vert\leq k^{-1}$ for all $n\geq n_k$. Now we set $R(n)=R_0$ for all $n < n_{R_0}$ and $R(n)=k$ for all $n_k\leq n < n_{k+1}$, $k\geq R_0$. This gives a function $n\mapsto R(n)$ satisfying
 \[
  R(n)\geq R_0 \mbox{ for all }n\in \NN, \quad \lim_{n\rightarrow\infty} R(n)=\infty \quad \mbox{and}\quad \lim_{n\rightarrow \infty}\frac{\vert \partial^{R(n)}Q_n\vert }{\vert Q_n\vert}=0.
 \]
Furthermore recall that $\epsilon=\epsilon(R)=\sum_{y\in G\setminus B_R} p(y)$, as in (\ref{def_epsilon(R)}). Thus we have $\lim_{n\rightarrow\infty} \epsilon(R(n))=0$. Setting $\delta(j):=(j^{1/4} \tau)^{-1}$ implies for fixed $n\in\NN$
\begin{equation}\label{theorem_erg3}
 \delta (j)\leq \tau^{-1} \mbox{ for all } j\in \NN, \quad 
\lim_{j\to\infty}  \delta(j) = 0
\quad\mbox{as well as}\quad
\exp\left(-\frac{\delta(j)^2|Q_j|}{4} \right)
\leq 
\exp\left(-\frac{j^{1/2}}{4\tau^{2}} \right)
\end{equation}
 for all $j\in\NN$. Here we used $j\leq \vert Q_j\vert$, which holds since $(Q_j)$ is strictly increasing.
Now for $j,n\in \NN$ Lemma \ref{lemma_unif} implies that
 \[
 D_\omega(j,n) := D_\omega(j,n,R(n))  \leq  B_\omega (j,n,R(n),\delta(j)) =: B_\omega (j,n)
 \]
 holds for all $\omega\in\Omega_j:=\tilde\Omega(\delta(j),R(n),Q_j)\cap \Omega_{\mathrm{fr}}$ and $\PP(\Omega_j)\geq 1-\exp(-j^{1/2}/4\tau^{2})$ by \eqref{theorem_erg3} and Corollary \ref{corollary_bernstein}.
 Furthermore for each $\omega\in\Omega_j$ we have 
 \[
 \lim_{n\rightarrow\infty}\lim_{j\rightarrow\infty}B_\omega(j,n)=0.
 \]
 Given $j,n\in \NN$ we set
\[
A_j^{(n)}:=\{\omega\in \Omega_{\mathrm{lf}}\cap \Omega_{\mathrm{fr}} \st D_\omega(j,n)>B_\omega(j,n)\}.
\]
Therefore $\PP(A_j^{(n)})\leq \exp(-  j^{1/2}/4\tau^2)$ for all $j\in\NN$ and hence $\sum_j \PP(A_j^{(n)}) <\infty$ holds. Applying Borel-Cantelli lemma leads to
\[
 \PP(A^{(n)})=0,\hspace{1cm} \mbox{where} \hspace{1cm} A^{(n)}:= \bigcap_{k=1}^\infty \bigcup_{j=k}^\infty A_j^{(n)}=\{ A_j^{(n)} \mbox{ infinitely often }\}.
\]
Thus we get
\[
 \PP\left(\omega \in \Omega_{\mathrm{lf}}\cap\Omega_{\mathrm{fr}} \ \Big\vert\ \lim_{j\rightarrow\infty}\left( D_\omega(j,n)-B_\omega(j,n)\right)\leq 0\right)=1
\]
for all $n\in \NN$.
And hence there exists a set $\tilde \Omega \subset \Omega_{\mathrm{lf}}\cap \Omega_{\mathrm{fr}}$ with $\PP(\tilde \Omega)=1$ such that
\[
 \lim_{n\rightarrow\infty}\lim_{j\rightarrow\infty}\left(D_\omega(j,n)-B_\omega(j,n)\right)\leq 0 \hspace{1cm}\mbox{for all }\omega \in \tilde \Omega
\]
which implies by definition of $B_\omega(j,n)$
\begin{equation}\label{theorem_erg1}
 \lim_{n\rightarrow\infty}\lim_{j\rightarrow\infty}  D_\omega(j,n) = 0 \hspace{1cm}\mbox{for all }\omega \in \tilde \Omega.
\end{equation}

Let $\kappa>0$ and $\omega\in\tilde\Omega$ arbitrary. There exists a natural number $n_0=n_0(\omega,\kappa)$ satisfying $\lim_{j\rightarrow\infty} D_\omega(j,n_0)\leq \kappa/8$, thus there exists $j_0=j_0(\omega,\kappa) \in \NN$ such that $D_\omega(j,n_0)\leq \kappa/4$ for all $j\geq j_0$. Using triangle inequality gives that 

\begin{eqnarray*}
\left \Vert  \frac{F_\omega^{R(n_0)}(Q_j)}{\vert Q_j \vert} - \frac{F_\omega^{R(n_0)}(Q_m)}{\vert Q_m \vert} \right\Vert 
&\leq& \left\Vert  \frac{F_\omega^{R(n_0)}(Q_j)}{\vert Q_j \vert} - \sum_{S\in {\mathcal S}(Q_{n_0})} \nu_S \frac{\tilde F^{R(n_0)}(S)}{\vert Q_{n_0}\vert}\right \Vert \\
&&+ \left\Vert\frac{F_\omega^{R(n_0)}(Q_m)}{\vert Q_m \vert} -\sum_{S\in {\mathcal S}(Q_{n_0})} \nu_S \frac{\tilde F^{R(n_0)}(S)}{\vert Q_{n_0}\vert}  \right\Vert \\
&\leq & D_\omega (j,{n_0}) + D_\omega (m,{n_0}) \leq \frac{\kappa}{2}
\end{eqnarray*}
holds for all $j,m\geq j_0$. Furthermore we use Lemma \ref{lemma_4} to obtain that there exists a $j_1=j_1(\kappa)\in \NN$ such that
\begin{equation}\label{theorem_erg2}
\left \Vert  \frac{F_\omega(Q_j)}{\vert Q_j \vert} - \frac{F_\omega^{R(n_0)}(Q_j)}{\vert Q_j \vert} \right\Vert 
=\left \Vert \frac{n(\Delta_\omega[Q_j])-n(\Delta_\omega[Q_{j,R(n_0)}])}{\vert Q_j\vert} \right \Vert
\leq \frac{4\vert \partial^{R(n_0)} Q_j \vert}{\vert Q_j\vert} 
\leq \frac{\kappa}{4}
\end{equation}
for all $j\geq j_1$. Now the triangle inequality yields
\begin{eqnarray*}
 \left \Vert  \frac{F_\omega(Q_j)}{\vert Q_j \vert} - \frac{F_\omega(Q_m)}{\vert Q_m \vert} \right\Vert &\leq & 
\left \Vert  \frac{F_\omega(Q_j)}{\vert Q_j \vert} - \frac{F_\omega^{R(n_0)}(Q_j)}{\vert Q_j \vert} \right\Vert 
+\left \Vert  \frac{F_\omega^{R(n_0)}(Q_j)}{\vert Q_j \vert} - \frac{F_\omega^{R(n_0)}(Q_m)}{\vert Q_m \vert} \right\Vert \\
&&+\left \Vert  \frac{F_\omega^{R(n_0)}(Q_m)}{\vert Q_m \vert} - \frac{F_\omega(Q_m)}{\vert Q_m \vert} \right\Vert \\
&\leq& \frac{\kappa}{4}  +\frac{\kappa}{2}+   \frac{\kappa}{4} =\kappa
\end{eqnarray*}
for all $j,m\geq \max \{j_0,j_1\}$, which implies that $\vert Q_j\vert^{-1} F_\omega(Q_j)$ is a Cauchy sequence and hence convergent in the Banach space $\cB(\RR)$ for all $\omega \in \tilde\Omega$. We denote the limit function by $N$.

It remains to show that $\sum_{S\in \cS(Q_n)}\nu_S \frac{\tilde F(S)}{\vert Q_n\vert}$ converges to the same limit. Therefore we fix $\omega\in\tilde\Omega$ and consider
\begin{equation*}
 \lim_{n\rightarrow \infty}\left\Vert N -  \sum_{S\in \cS(Q_n)}\nu_S \frac{\tilde F(S)}{\vert Q_n\vert} \right \Vert
= \lim_{n\rightarrow \infty}\lim_{j\rightarrow\infty} \left\Vert \frac{F_\omega(Q_j)}{\vert Q_j\vert } -  \sum_{S\in \cS(Q_n)}\nu_S \frac{\tilde F(S)}{\vert Q_n\vert} \right \Vert.
\end{equation*}
Adding zeros leads to the inequality
\begin{multline}
 \left\Vert \frac{F_\omega(Q_j)}{\vert Q_j\vert } -  \sum_{S\in \cS(Q_n)}\nu_S \frac{\tilde F(S)}{\vert Q_n\vert} \right \Vert
\leq  \left\Vert \frac{F_\omega(Q_j)}{\vert Q_j\vert } - \frac{F_\omega^{R(n)}(Q_j)}{\vert Q_j\vert }  \right \Vert
\\ + \left\Vert \frac{F_\omega^{R(n)}(Q_j)}{\vert Q_j\vert }  -  \sum_{S\in \cS(Q_n)}\nu_S \frac{\tilde F^{R(n)}(S)}{\vert Q_n\vert} \right \Vert  
+ \left\Vert \sum_{S\in \cS(Q_n)}\nu_S \frac{\tilde F^{R(n)}(S)}{\vert Q_n\vert} -  \sum_{S\in \cS(Q_n)}\nu_S \frac{\tilde F(S)}{\vert Q_n\vert} \right \Vert.
\end{multline}
Now we take $\lim_{n\rightarrow\infty}\lim_{j\rightarrow\infty}$ on both sides and obtain that the three summands on the right vanish. The first one 
is zero by an estimate as in (\ref{theorem_erg2}). Applying (\ref{theorem_erg1}) gives that the second summand vanishes. The third summand tends to zero since Lemma \ref{lemma_4} yields
\begin{eqnarray*}
\left\Vert \sum_{S\in \cS(Q_n)}\nu_S \frac{\tilde F^{R(n)}(S)}{\vert Q_n\vert} -  \sum_{S\in \cS(Q_n)}\nu_S \frac{\tilde F(S)}{\vert Q_n\vert} \right \Vert 
&\leq&
\sum_{S\in \cS(Q_n)}\nu_S \frac{\left\Vert \tilde F^{R(n)}(S)-\tilde F(S)  \right\Vert}{\vert Q_n\vert} \\
&\leq&
\sum_{S\in \cS(Q_n)}\nu_S \frac{4\left\vert \partial^{R(n)}Q_n  \right\vert}{\vert Q_n\vert}
\end{eqnarray*}
and for some fixed $y\in Q_n$
\begin{align*}
 \sum_{S\in \cS(Q_n)}\nu_S 
&= \lim_{j\to\infty}\frac{1}{|Q_j|} \sum_{S\in \cS(Q_n)} \vert \{x\in G \st V_S x\subset Q_j ,\ \Gamma_\omega[V_Sx]= Sx \}\vert\\
&\leq \lim_{j\to\infty}\frac{1}{|Q_j|} \sum_{S\in \cS(Q_n)} \vert \{z \in Q_j \st x:=y^{-1}z,\  \Gamma_\omega[V_Sx]= Sx \}\vert\\
&= \lim_{j\to\infty}\frac{1}{|Q_j|} \left\vert\dot\bigcup_{S\in \cS(Q_n)}  \{z \in Q_j \st x:=y^{-1}z,\   \Gamma_\omega[V_Sx]= Sx \}\right\vert
\leq 1 .
\end{align*}
This proves the claimed convergence for all $\omega \in \tilde \Omega$.
\end{proof}

\section{Discontinuities}
In this section we investigate the points of discontinuity of the integrated density of states. We firstly prove a criteria that the IDS has a jump at $\lambda\in \RR$. Afterwords we characterise the set of points of discontinuity as a large subset of the real axis.

\begin{Theorem}\label{thm_dis}
 There exists a set $\tilde\Omega\subset \Omega$ of full measure such that for each $\omega\in \tilde\Omega$ and $\lambda \in \RR$ the following assertions are equivalent:
\begin{itemize}
 \item [(a)] $\lambda$ is a point of discontinuity of $N$
 \item [(b)] there exists a finitely supported eigenfunction corresponding to $\lambda$
 \item [(c)] there exist infinitely many mutually independent finitely supported eigenvectors corresponding to $\lambda$
\end{itemize}
\end{Theorem}
\begin{proof}
Let $(Q_j)$ be a strictly increasing, tempered F\o lner sequence and $\tilde\Omega\subset \Omega$ a set of full measure such that Theorem \ref{theorem_erg} holds for all $\omega \in\tilde\Omega$. Note that $\tilde \Omega\subset \Omega_{\mathrm{fr}}\cap\Omega_{\mathrm{lf}}$, which implies in particular that for an arbitrary graph $S\in\cS$ and $\omega\in \tilde \Omega$ the frequency $\nu_S$ in $\Gamma_\omega$ along $(Q_j)$ exists. As $p$ is assumed to be an element of $\ell^1(G)$ there exists $R\in\NN$ such that $p(xy^{-1})$ is strictly smaller than $1$ for all $x,y\in G$ satisfying $d(x,y)\geq R$. We fix this $R\in \NN$ and some $\omega \in \tilde\Omega$.

Let $\lambda$ be a point of  discontinuity of $N$. Theorem \ref{theorem_erg} yields that $n(\Delta_\omega[Q_j])/\vert Q_j\vert$ approximates the IDS $N$ uniformly in the energy variable. Hence there exists a constant $c>0$ such that
\[
\dim(\ker(\Delta_\omega[Q_j]-\lambda))
=\lim_{\epsilon \to \infty}\left(n(\Delta_\omega[Q_j])(\lambda+\epsilon)-n(\Delta_\omega[Q_j])(\lambda-\epsilon)\right)
\geq c\vert Q_j\vert.
\]
for all $j\in\NN$. Since $(Q_j)$ is a F\o lner sequence, we have $\lim_{j\to \infty}\vert\partial_{\mathrm{int}}^{R}Q_j\vert /\vert Q_j\vert=0$, which implies the existence of $k\in \NN$ such that
\[
 \dim(\ker(\Delta_\omega[Q_k]-\lambda))
\geq c\vert Q_k\vert
> \vert \partial_{\mathrm{int}}^R Q_k \vert
=\dim(\ell^2(\partial_{\mathrm{int}}^R Q_k ))
\]
holds. A well known dimension argument yields that there exists an element $0\neq u\in\ell^2(Q_k)$ satisfying $(\Delta_\omega[Q_k]-\lambda) u=0$ and $u\equiv 0$ on $\partial_{\mathrm{int}}^R Q_k $. Now we consider the subgraph 
\begin{equation}\label{thm_dis2}
 S:=(V_S,E_S):=\Gamma_\omega[Q_{k}].
\end{equation}
 Lemma \ref{lemma_nu} proves that the frequency of $R$-isolated occurrences of $S$ in $\Gamma_\omega$ along $(Q_j)$ is given by
\begin{equation}\label{thm_dis1}
\nu_{S,R} =  \prod_{[x,y]\in E_S} p(xy^{-1}) \cdot \prod_{\ato{[x,y]\notin E_S}{x,y\in V_S}} (1-p(xy^{-1})) \cdot \prod_{\ato{[x,y]\in E, x\in V_S,}{y\notin  V_S, d(x,y)\geq R }} (1-p(xy^{-1}))
 \end{equation}
Here the first two products have to be non-zero as $S$ is a restriction of $\Gamma_\omega$. The positivity of the infinite product follows from the choice of $R$ and the summability condition on $p$. 
This implies that there is an infinite set $M\subset G$ such that $\Gamma_\omega[Q_k x]$ is an $R$-isolated copy of $S$ for each $x\in M$. Furthermore there exists an infinite subset $M'\subset M$ such that $Q_kx\cap Q_ky=\empty$ for all $x,y\in M'$. For $x\in M'$ we define $u_x\in \ell^2(G)$ by setting
\[
 u_x(g)=\begin{cases}u(gx^{-1})& g\in Q_kx\\ 0&\mbox{else.} \end{cases}
\]
Then $u_x$, $x\in M'$ are mutually independent, finitely supported eigenfunctions of $\Delta_\omega$ corresponding to $\lambda$. This proves that (a) implies (c).

Obviously (c) implies (b), thus it remains to show that given a finitely supported eigenfunction $u$ corresponding to $\lambda \in \RR$ the IDS is discontinuous at $\lambda$. To this end let $r>0$ be large enough that $\supp(u)\subset B_r$. As $\omega\in \Omega_{\mathrm{lf}}$ the graph $\Gamma_\omega$ is locally finite. Therefore we find $s>r$ such that there are no edges connecting the sets $B_r$ and $G\setminus B_s$ in $\Gamma_\omega$. Now we consider the graph $S=(V_S,E_S):=\Gamma_\omega[B_{t}]$, where $t:=s+R$. As $S$ is a restriction of $\Gamma_\omega$ the frequency $\nu_{S,R}$ of $R$-isolated occurrences of $S$ in $\Gamma_\omega$ along $(Q_j)$ is strictly positive. Thus there exists a constant $c>0$ such that $\sharp_{S,R}(\Gamma_\omega,Q_j)\geq c\vert Q_j\vert$ for $j$ large enough.

For given $Q\in \cF(G)$ each disjoint $R$-isolated copy of $S$ in $\Gamma_\omega[Q]$ adds a dimension to the eigenspace of $p_Q \Delta_\omega i_Q$ corresponding to $\lambda$. Therefore we define $\dot\sharp_{S,R}(\Gamma_\omega,Q)$ to be the maximal number of disjoint and $R$-isolated occurrences of the subgraph $S$ in $\Gamma_\omega[Q]$. It is easy to verify that in this situation the inequality $\vert B_{3t}\vert\dot\sharp_{S,R}(\Gamma_\omega,Q) \geq\sharp_{S,R}(\Gamma_\omega,Q)$ holds.
For each $\epsilon>0$ we get
\[
 \frac{n(\Delta_\omega[Q])(\lambda-\epsilon)}{\vert Q\vert} 
\leq \frac{n(\Delta_\omega[Q])(\lambda + \epsilon)-\dot\sharp_{S,R}(\Gamma_\omega,Q)}{\vert Q\vert}
\leq \frac{n(\Delta_\omega[Q])(\lambda + \epsilon)}{\vert Q\vert} -\frac{\sharp_{S,R}(\Gamma_\omega,Q)}{\vert B_{3t}\vert \vert Q\vert}.
\]
Replacing $Q$ by elements of the sequence $(Q_j)$ yields
\[
 \frac{n(\Delta_\omega[Q_j])(\lambda + \epsilon)}{\vert Q_j\vert}-\frac{n(\Delta_\omega[Q_j])(\lambda - \epsilon)}{\vert Q_j\vert}
\geq 
\frac{\sharp_{S,R}(\Gamma_\omega,Q_j)}{\vert B_{3t}\vert \vert Q_j\vert}.
\]
We let $j$ tend to infinity and obtain
\[
 N(\lambda+\epsilon) - N(\lambda-\epsilon)\geq \frac{\nu_{S,R}}{\vert B_{3t}\vert},
\]
which proves that $\lambda$ is a point of discontinuity of $N$.
\end{proof}

Now we study the set of points of discontinuity, which obviously depends on the specific choice of the function $p\in \ell^1(G)$. Here we consider the case where the given function $p$ satisfies not just (\ref{def_p}) but even
\begin{equation}\label{def_pnew}
 0<p(x)<1    \hspace{0.7cm}\mbox{and}\hspace{0.7cm}   p(x)=p(x^{-1})
\end{equation}
for all $x\in G$. Defining the set
\[
 W=\{\lambda\in \RR \st \exists\ S\in \cS \mbox{ with } \lambda \in \sigma(\Delta_S )\},
\]
we prove the following

\begin{Corollary}\label{cor_disc}
Let $p \in \ell^1(G)$ satisfying (\ref{def_pnew}) and the associated probability space $(\Omega,\cA,\PP)$ be given. Then the set of points of discontinuity of the IDS $N$ equals $W$ almost surley.
\end{Corollary}

\begin{proof}
 Let $\tilde\Omega \subset \Omega_{\mathrm{fr}}\cap\Omega_{\mathrm{lf}}$ be a set of full measure such that Theorem \ref{theorem_erg} holds for all $\omega \in\tilde\Omega$  and choose some $\omega\in\tilde\Omega$.

 Let $\lambda$ be a point of discontinuity of $N$. By Theorem \ref{thm_dis} there is a finitely supported eigenfunction $u$ corresponding to $\lambda$. As in the proof of Theorem \ref{thm_dis} we find $r>0$ such that $\supp(u)\subset B_r$ and $s>r$ such that there are no edges in $\Gamma_\omega$ connecting $B_r$ with $G\setminus B_s$. We set $S=(V_S,E_S)=\Gamma_\omega[B_s]$. Therefore $\lambda$ is an eigenvalue of $\Delta_S$ with eigenfunction $ p_{V_S} u$.
 
 Let $\lambda$ be an element in $W$, i.e. there exists $S=(V_S,E_S)\in \cS$ such that $\lambda$ is an eigenvalue of the associated Laplacian $\Delta_S$. Let $u$ be an associated eigenfunction. By Lemma \ref{lemma_nu} the frequency $\nu_{S,1}$ is given by 
 \[
  \nu_{S,1}=  \prod_{[x,y]\in E_S} p(xy^{-1}) \cdot \prod_{\ato{[x,y]\notin E_S}{x,y\in V_S}} (1-p(xy^{-1})) \cdot \prod_{\ato{[x,y]\in E, x\in V_S,}{y\notin  V_S, d(x,y)\geq 1}} (1-p(xy^{-1}))
 \]
which is strictly positive by assumption on $p$. Thus there exists a $x\in G$ such that $Sx$ is a $1$-isolated copy of $S$ in $\Gamma_\omega$. Then $u'\in \ell^2(G)$ given by
\[
 u'(g)=\begin{cases} u(gx^{-1}) & g\in V_S x \\ 0 &\mbox{else} \end{cases}
\]
is a finitely supported eigenfunction of $\Delta_\omega$ corresponding to $\lambda$. By Theorem \ref{thm_dis} this implies the discontinuity of $N$ at $\lambda$.
\end{proof}

\subsection*{Acknowledgement} The author would like to thank his advisor Ivan Veseli\'c for suggesting him the research topic of this paper. He would also like to thank Daniel Lenz for most valuable hints concerning the usability of Chebyshev inequality. These remarks made proofs more elegant and put the results in a larger generality.

\bibliographystyle{alpha}
\newcommand{\etalchar}[1]{$^{#1}$}
\def\cprime{$'$}\def\polhk#1{\setbox0=\hbox{#1}{\ooalign{\hidewidth
  \lower1.5ex\hbox{`}\hidewidth\crcr\unhbox0}}}

\end{document}